\documentclass[12pt, letterpaper]{amsart}

\newcommand{\indentalign}{\hspace{0.3in}&\hspace{-0.3in}}
\newcommand{\la}{\langle}
\newcommand{\ra}{\rangle}

\newcommand{\defeq}{\stackrel{\rm{def}}{=}}
\newcommand{\supp}{\operatorname{supp}}

\newcommand{\sgn}{\operatorname{sgn}}
\newcommand{\pv}{\operatorname{pv}}

\usepackage{amssymb}
\usepackage{amsthm}
\usepackage{amsxtra}
\usepackage{graphicx}

\newtheorem{theorem}{Theorem}

\newtheorem{proposition}[theorem]{Proposition}
\newtheorem{lemma}[theorem]{Lemma}

\theoremstyle{remark}

\numberwithin{equation}{section}

\numberwithin{theorem}{section}

\numberwithin{table}{section}

\numberwithin{figure}{section}

\ifx\pdfoutput\undefined
  \DeclareGraphicsExtensions{.pstex, .eps}
\else
  \ifx\pdfoutput\relax
    \DeclareGraphicsExtensions{.pstex, .eps}
  \else
    \ifnum\pdfoutput>0
      \DeclareGraphicsExtensions{.pdf}
    \else
      \DeclareGraphicsExtensions{.pstex, .eps}
    \fi
  \fi
\fi

\title[Benjamin-Ono soliton dynamics in a slowly varying potential]{Benjamin-Ono soliton dynamics in a \\ slowly varying potential}

\author{Katherine Zhiyuan Zhang}
\address{Brown University}

\begin{document}

\maketitle

\begin{abstract}
The Benjamin Ono equation with a slowly varying potential is
$$
\text{(pBO)} \qquad u_t + (Hu_x-Vu + \tfrac12 u^2)_x=0
$$
with $V(x)=W(hx)$, $0< h \ll 1$, and $W\in C_c^\infty(\mathbb{R})$, and $H$ denotes the Hilbert transform.  The soliton profile is $Q_{a,c}(x) = cQ(c(x-a))$, where $Q(x) = \frac{4}{1+x^2}$ and $a\in \mathbb{R}$, $c>0$ are parameters.   For initial condition $u_0(x)$ to (pBO) close in $H_x^{1/2}$ to $Q_{0,1}(x)$, we show that the solution $u(x,t)$ to (pBO) remains close in $H_x^{1/2}$ to $Q_{a(t),c(t)}(x)$ and specify the $(a,c)$ parameter dynamics on an $O(h^{-1})$ time scale. 
\end{abstract}

\section{Introduction}
\label{S:intro}

Let $H$ be the Hilbert transform, corresponding to the Fourier multiplier $i\sgn \xi$.   (For further elaboration on notational conventions, see \S \ref{A:notation}.) 
We consider the Benjamin-Ono equation (BO)
$$
\text{(BO)} \qquad 
\partial_t u = \partial_x (-H\partial_x u - \frac12 u^2)
$$
with $u$ real-valued, on $\mathbb{R}$.  The equation (BO) is a model for 1D long internal waves in stratefied fluid, introduced by Benjamin \cite{Ben} and Ono \cite{Ono}.    (BO) has much in common with the Korteweg-de Vries equation (KdV)
$$\text{(KdV)} \qquad \partial_t u = \partial_x (-\partial_x^2 u - 3 u^2)$$
such as physical origin (KdV is a model of waves on shallow water surfaces) and the mathematical structure of complete integrability.  Notably for our purposes as we discuss below, (KdV) has the same symplectic structure as (BO), when regarded as a Hamiltonian system, and both (KdV) and (BO) possess single solitary waves that propagate to the right.  A key difference is that (BO) involves the Hilbert transform, which is a \emph{nonlocal} operator, and this leads to solitary waves that have only algebraic decay for (BO), as opposed to exponential decay for (KdV).

By working with the three transformations $u(x,-t)$, $u(-x,t)$, and $-u(x,t)$ we are in fact covering all four sign choices in $\partial_t u = \partial_x (\pm H\partial_x u \pm \frac12 u^2)$.   Hence we do not have a distinction between ``focusing" or ``defocusing" problems.   Moreover, (BO) also satisfies translational invariance in space and has the scaling invariance, for $\lambda>0$,
$$u \text{ solves (BO)} \implies u_\lambda(x,t) = \lambda u(\lambda x, \lambda^2 t) \text{ solves (BO)}$$
(BO) is completely integrable, so it enjoys infinitely many conserved quantities \cite{BK}, the first three of which are
$$
M_0(u) = \frac12 \int u^2\,, \quad E_0(u) = -\frac12\int uHu_x - \frac{1}{6} \int u^3, $$
$$E_1(u) = \frac12\int u^2_x +\frac{3}{8}\int u^2 Hu_x -\frac{1}{16} \int u^4 $$

Tao \cite{Tao} proved local well-posedness of (BO) in $H_x^1$, and global well-posedness follows using the aforementioned conserved quantities.  This result followed several earlier results at higher regularity, including \cite{Saut, Iorio, GV, Ponce, KT, KK}.  The innovation Tao introduced was a gauge transformation to reduce the effective regularity of the nonlinearity. Following \cite{Tao}, there were a few improvements to even lower regularity, using the gauge transformation idea combined with bilinear Strichartz estimates, culminating in the $L^2$ result by Ionescu \& Kenig \cite{IK} and Molinet \& Pilod \cite{MolPil}.

By \emph{soliton} we mean a coherent traveling wave solution.  Amick \& Toland \cite{AT}, Frank \& Lenzmann \cite{FrLen} showed that there is a unique (up to translations) nontrivial $L^\infty$ solution to 
$$Q -HQ' - \frac12 Q^2 =0$$
given by
$$Q(x) = \frac{4}{1+x^2}$$
For any $c>0$, $a\in \mathbb{R}$, taking $Q_{a,c}(x) = cQ(c(x-a))$ we have
\begin{equation}
\label{E:scaled-sol}
cQ_{a,c} - HQ_{a,c}' - \frac12 Q^2_{a,c} =0
\end{equation}
Then
$$u(x,t) = Q_{ct,c}(x) = cQ(c(x-ct))$$
solves (BO) and we call it the \emph{single soliton} solution to distinguish it from the exact multi-soliton solutions \cite{Case} arising from the completely integrable structure.   The (BO) soliton is only decaying at infinity at power rate unlike for (KdV) where the soliton enjoys exponential decay.
 
Having summarized the basic properties of (BO), we now consider the following Hamiltonian perturbation of (BO)
$$
\text{(pBO)} \qquad
\partial_t u =  \partial_x (- H\partial_x u + V u - \frac12 u^2) 
$$
with \emph{slowly varying} potential
\begin{equation}
\label{E:potential}
V(x) = W(hx) \,, \qquad W\in C_c^\infty(\mathbb{R}) \text{ and } 0<h \ll 1
\end{equation}
The well-posedness of (pBO) in $H^1$ can be proved by adapting the gauge-transform method of Tao \cite{Tao}.   The Hamiltonian has been perturbed to
$$E(u) = E_0(u) + \frac12 \int Vu^2$$
(pBO) is of the form $\partial_t u = JE'(u)$, where $J = \partial_x$.  The symplectic form is given by $\omega(v_1,v_2) = \la v_1, J^{-1}v_2 \ra$, which is only densely defined.\footnote{Any integrable function $v_2$ for which $\int_{-\infty}^{+\infty} v_2(x) \, dx \neq 0$ has the property that $\partial_x^{-1} v_2 \notin L_x^2$.  Hence the symplectic form is only densely defined. Despite this fact, the symplectic projection onto the soliton manifold $\mathcal{M}$ is well-defined due to the fact that the tangent space of $\mathcal{M}$ is spanned by functions that are the derivatives of smooth functions in $L^2$.}   The restriction (pull-back by the inclusion map $i$) of this symplectic form to the two-dimensional \emph{soliton manifold} 
$$\mathcal{M} = \{ \, Q_{a,c} \, | \, a\in \mathbb{R} \,,  \; c>0 \; \}$$ 
is the canonical form
$$i^*(\omega) = 4\pi da \wedge dc$$
A heuristic states that \emph{if we assume} that the solution remains close to $\mathcal{M}$, then the projected flow on $\mathcal{M}$ follows the Hamilton equations corresponding to $i^*(\omega)$
\begin{equation}
\label{E:ODEs-1}
\left\{
\begin{aligned}
&\dot c = \frac{1}{4\pi} \partial_a E(Q_{a,c}) = chW'(ha) + \frac12 c^{-1}h^3W'''(ha) +O(h^5)\\
&\dot a = -\frac{1}{4\pi} \partial_c E(Q_{a,c}) = c- W(ha) + \frac12 c^{-2} h^2 W''(ha) + O(h^4)
\end{aligned}
\right.
\end{equation}

This ODE system in $( a, c)$, scales to an $h$-independent system at leading order.  Let $s= ht$, and consider $ A(s)$ and $ C(s)$ defined by
\begin{equation}
\label{E:convert}
A(ht) = h a(t) \,, \qquad C(ht) =  c(t)
\end{equation}
Then the above becomes
\begin{equation}
\label{E:ODEs-2}
\left\{
\begin{aligned}
&\dot { C} =  CW'( A)+ \frac12 C^{-1}h^2W'''(A) + O(h^4) \\
&\dot { A} =  C- W( A) +\frac12 C^{-2} h^2 W''(A) +O(h^4)
\end{aligned}
\right.
\end{equation}
We have
$$0 \leq s \leq 1 \quad \iff   \quad 0 \leq t \leq h^{-1}$$
The time scale $0 \leq s \ll \ln h^{-1}$ corresponds to $0\leq t \ll\delta h^{-1}\ln h^{-1}$, which we call the \emph{Ehrenfest time}.  Analytically, we will \emph{prove} that a solution starting close to $\mathcal{M}$ remains close to $\mathcal{M}$ for the Ehrenfest time and recover ODEs describing the parameter motion with slightly less precision than \eqref{E:ODEs-1} and \eqref{E:ODEs-2}.

For the statement of our main result, we will need the \emph{reference trajectory}, which is the solution $(\bar A(s), \bar C(s))$ to 
\begin{equation}
\label{E:ref-traj}
\left\{
\begin{aligned}
&\dot {\bar C} =  \bar CW'(\bar A) \\
&\dot {\bar A} =  \bar C- W( \bar A) 
\end{aligned}
\right.
\end{equation}
with initial condition $(\bar A(0),\bar C(0)) = (0,1)$.  This is an $h$-independent system and captures the leading-order $h$-independent terms in the expected full parameter dynamics \eqref{E:ODEs-2}.    We let $S_0>0$ be the first time $s>0$ such that  $\bar C(s)=\frac12$ or $\bar C(s)=2$, and $S_0=+\infty$ if $C(s)$ never reaches either $\frac12$ or $2$.    Thus, for all $0\leq s < S_0$, we have
$$\frac12 \leq \bar C(s) \leq 2$$
Now let
\begin{equation}
\label{E:bar-convert}
\bar a(t)=h^{-1}\bar A(ht)\,, \qquad \bar c(t) = \bar C(ht)
\end{equation}
Throughout the paper, the notation $A \lesssim B$ means there exists a constant $K >0$ such that $A\leq KB$.

\begin{theorem}[main theorem]
\label{T:main}
Given a potential $W\in C_c^\infty(\mathbb{R})$ (as in \eqref{E:potential}), there exists $\kappa \geq 1$, $\mu>0$, and $0<h_0\ll 1$ such that the following holds.  Let $0< h \leq h_0$ and suppose the initial data $u_0\in H_x^1$ satisfies
$$\| u_0(x) - Q_{0,1}(x) \|_{H_x^{1/2}} \leq h^{3/2}$$
Then there exists a trajectory $(a(t),c(t))$ such that $u$ solving (pBO) with initial condition $u_0$ satisfies
\begin{equation}
\label{E:control}
\| u(x,t) - Q_{a(t),c(t)}(x) \|_{H_x^{1/2}} \leq \kappa h^{3/2}e^{\mu ht}
\end{equation}
for $0\leq t \leq T_0=h^{-1}\min( \frac14\mu^{-1} \ln h^{-1}, S_0)$.  

We have the following information about the trajectory $(a(t),c(t))$. Let $(\bar A(s),\bar C(s))$ solve \eqref{E:ref-traj}
on $0\leq s \leq  \min(\frac14\mu^{-1} \ln h^{-1},S_0)$, with initial condition
$(\bar A(0),\bar C(0))=(0,1)$, and then define $(\bar a(t), \bar c(t))$ according to \eqref{E:bar-convert}.  Then we have
\begin{equation}
\label{E:true-vs-ref}
|a(t)-\bar a(t)| \lesssim h e^{2\mu h t}\,, \qquad |c(t)-\bar c(t)| \lesssim h^2 e^{2\mu h t}
\end{equation}
\end{theorem}
Notice that \eqref{E:true-vs-ref} implies that we can replace $c(t)$ by $\bar c(t)$ in \eqref{E:control}, but cannot replace $a(t)$ by $\bar a(t)$ unless we sacrifice $h^{3/2}$ accuracy for $h$ accuracy in \eqref{E:control}.

Let us remark on the $O(h^{-1})$ time scale in Theorem \ref{T:main}.  The symplectic restriction heuristic produces the expected ODEs \eqref{E:ODEs-1}.  The rescaling of time given by $s=ht$ converts this ODE system to \eqref{E:ODEs-2}, which has the feature that the leading order terms in both components are \emph{independent of} $h$ and \emph{nontrivial perturbations of the free soliton dynamics} (involve $W$).  Then achieving time $s=O(1)$ (that is, $t=O(h^{-1})$) lets us observe the nontrivial distortions of the position $a$ and scale $c$ parameters.    This makes $O(h^{-1})$ a dynamically relevant time frame as $h\to 0$.

The key technical device needed to control the evolution is the Lyapunov functional $L_c(u)$ appearing in \eqref{E:Lyapunov}.  We are not able to extend Theorem \ref{T:main} beyond the time scale $O(h^{-1}\ln h^{-1})$ due to error terms that appear in the time derivative of $L_c(u)$.  Specifically, at least one term arises that is comparable  to $hL_c(u)$, so that the best possible estimate is
$$|\partial_t L_c(u) | \lesssim h L_c(u)$$
which results in the bound $L_c(u(t)) \lesssim e^{ht} L_c(u(0))$, and this bound is only useful slightly beyond the time scale $O(h^{-1})$.

Let us now provide an overview of related results.  Fr\"ohlich et. al. \cite{FGJS} and Holmer \& Zworski \cite{HZ2} considered Hamiltonian semiclassical perturbations of the 1D nonlinear Schr\"odinger (NLS) equation
$$i\partial_t u + \Delta u - W(hx) u + |u|^2 u =0$$
In \cite{HZ2}, solutions are shown to remain $h^2$ close to a solitary wave profile in the energy space $H^1$ on the time scale $\delta h^{-1} \log h^{-1}$.    Datchev \& Ventura \cite{DV} treated the case of the Hartree nonlinearity, and de Bouard \& Fukuizumi \cite{deBF} considered a stochastic perturbation of NLS.  Fractional NLS equations have been considered by Secchi \& Squassina \cite{SS}, and a variational approach has been employed to study NLS in Benci, Ghimenti \& Micheletti \cite{BGM}.   Holmer \& Lin \cite{HL} considered a related problem of the interaction of two overlapping solitons, and Holmer \& Zworski \cite{HZ1} considered the dynamics of near soliton solutions to an NLS equation perturbed by a weak delta potential.

Dejak \& Sigal \cite{DS} and Holmer \cite{Holmer} studied Hamiltonian semiclassical perturbations of the Korteweg-de Vries (KdV) equation
$$\partial_t u + \partial_x (\partial_x^2 u -W(hx) u +  \partial_x u^2) =0$$
In \cite{Holmer}, solutions are shown to remain $h^{1/2}$ close to a solitary wave profile in the energy space $H^1$ on the time scale $\delta h^{-1} \log h^{-1}$.  A related result pertaining to a nonlinear perturbation of KdV was obtained by Mu\~noz \cite{Mun}.  Lin \cite{Lin} considered a nonHamiltonian perturbation of mKdV.    Dynamics of near double solitons for mKdV under semiclassical perturbation of mKdV was studied by Holmer, Perelman, \& Zworski \cite{HPZ}.  De Bouard and Debussche \cite{deBD} considered a stochastic perturbation of KdV with multiplicative white noise.  

Pocovnicu \cite{Poc, PocErratum} and G\'erard \& Grellier \cite{GerGre} considered the cubic-Szego equation.  Mashkin \cite{Mash1, Mash2, Mash3} has considered perturbations of sine-Gordon kink solitons.    Heuristics on solitary wave perturbation for BO were previously obtained by Matsuno \cite{Matsuno} in two settings -- the BO Burgers equation (adding weak dissipation) and a BO equation with the inclusion of higher-order nonlinear terms.   

Aside from well-posedness issues \cite{Tao, Saut, Iorio, GV, Ponce, KT, KK, MolPil} and solitary wave stability \cite{GTT, KM}, the BO equation and its generalizations have recently been investigated from the point of view of inverse scattering and integrability by Wu \cite{Wu} and  G\'erard \& Kappeler \cite{GerKap}, and singularity formation (blow-up) by Martel \& Pilod \cite{MarPil}.

Theorem \ref{T:main} in the present paper appears to be the first rigorous result of this type for BO.  In comparison to the papers for NLS and KdV \cite{HZ2, Holmer}, the key difficulty results from the slow decay of the soliton $Q$ (at the power rate $x^{-2}$ as $|x|\to +\infty$).  At several points in the argument, the potential $W(ha+h(x-a))$ is Taylor-expanded, producing powers of $(x-a)$, although the soliton and its derivatives can only absorb the first few powers of the Taylor expansion, limiting expansions to quadratic or cubic order.   The issue arises in the two main estimates of the paper appearing in Lemma \ref{L:ODEcontrol} and Lemma \ref{L:energycontrol}.   The treatment of Term III' in \eqref{E:odestuff1} in the proof of Lemma \ref{L:ODEcontrol} is achieved by decomposing the spatial region into $|x-a|<h^{-1}$ and $|x-a|>h^{-1}$, and the Taylor expansion is only applied in the inner region, but to fourth order.  Also, in the treatment of Term III in the proof of Lemma \ref{L:energycontrol}, the Taylor expansion is limited to third order and terms resulting from the remainder in Taylor's formula are handled using the estimate \eqref{E:q-bound}.  Another difficulty, in comparison to the NLS and KdV results, is that several terms in the estimates involve a Hilbert transform.  Two lemmas in \S\ref{S:estimates} are given to handle such terms.  For example, they are applied in the treatment of Term III in Lemma \ref{L:energycontrol}.

We now give an overview of the organization of the paper.
In \S\ref{A:notation}, we state our notational conventions and review the definition and basic properties of the soliton profile $Q_{a, c}$ and its linearization $\mathcal{L}$.  In \S\ref{S:heuristics}, we give a heuristic derivation of the soliton dynamics \eqref{E:ODEs-1} following the principle of symplectic restriction, as previously described in \cite{HZ1, HL}.  In \S\ref{S:estimates}, we provide two estimates for quadratic forms involving the Hilbert transform and cutoffs, that are needed later in the proof of Lemma \ref{L:ODEcontrol}.  In \S\ref{S:energybd}, we state and prove Lemma \ref{L:en}, the spectral lower bound on $\mathcal{L}$, following ideas of Weinstein \cite{Wei} and Fr\"ohlich et. al. \cite{FGJS}, and using the explicit spectral resolution of $\mathcal{L}$ provided in the Appendix of Bennett et. al \cite{BBSSB}.  At this point, all of the technical ingredients for the proof of Theorem \ref{T:main} are in place.  In \S\ref{S:setup}, the proof of Theorem \ref{T:main} is reduced to Prop. \ref{P:main}, where the approximate ODEs \eqref{E:ODE12} replace the comparison to exact reference ODEs in \eqref{E:true-vs-ref}.  The proof that Prop. \ref{P:main} implies Theorem \ref{T:main} involves invoking an elementary Gronwall type estimate, given as Lemma \ref{L:gronwall}.  In the remainder of the paper, the proof of Prop. \ref{P:main} is given.  It consists of two main estimates.  First, in \S\ref{S:ODEcontrol}, control on the parameter ODEs is obtained, assuming control on the remainder function, by computing the derivatives of the orthogonality conditions.  Second, in \S\ref{S:remainder}, the remainder function is controlled assuming the parameter ODEs hold with sufficient accuracy.  This is accomplished by taking the derivative of the Lyapunov function, and appealing to Lemma \ref{L:en}.  The two lemmas Lemma \ref{L:ODEcontrol} and Lemma \ref{L:energycontrol}, coupled together, complete the proof of Prop. \ref{P:main}, which is written following the statement of Lemma \ref{L:energycontrol} in \S\ref{S:remainder}.

\subsection{Acknowledgment}
This project began while the author was reading papers \cite{KM, Tao, HZ2, Holmer} as part of her preparation for the ``topics exam'' at Brown in 2015-2016, under the direction of Justin Holmer.  The author is grateful for his assistance on this project at that time.

\section{Notation, soliton profile and linearization properties}
\label{A:notation}

We take the Fourier transform in 1D as
$$\hat f(\xi) = \int e^{-ix\xi} f(x) \, dx$$
and inverse Fourier transform
$$\check g(x) = \frac{1}{2\pi} \int e^{ix\xi} g(\xi)\, d\xi$$
We define the Hilbert transform as
$$Hf(x) = -\frac{1}{\pi} \pv \int_{-\infty}^{+\infty} \frac{f(y)}{x-y} \, dy = -\frac{1}{\pi} \pv \frac{1}{x} * f$$
and hence
$$\widehat{Hf}(\xi) = i (\sgn \xi) \hat f(\xi)$$
The fractional derivative operator $D^s$ is defined as
$$\widehat{D^sf}(\xi) = |\xi|^s \hat f(\xi)$$
and thus $-H\partial_x = D$.   Define the soliton profile (in standard position and scale) as
\begin{equation}
\label{E:D2}
Q(x) = \frac{4}{1+x^2}
\end{equation}
We have the partial fraction decomposition
$$\frac{1}{1+y^2} \frac{1}{x-y} = - \frac{1}{1+x^2}\frac{1}{y-x} + \frac{x}{1+x^2} \frac{1}{1+y^2} + \frac{1}{1+x^2} \frac{y}{1+y^2}$$
and hence (since first and third term integrate to zero)
$$HQ(x) = -\frac{4}{\pi} \frac{x}{1+x^2} \int \frac{dy}{1+y^2} = \frac{-4x}{1+x^2}$$
Taking the derivative in $x$ we have
$$HQ_x = -\frac{8}{(1+x^2)^2} + \frac{4}{1+x^2} = -\frac12Q^2+Q$$
which is the soliton profile equation, which we rewrite for convenient reference:
\begin{equation}
\label{E:D3}
-HQ_x + Q - \frac12 Q^2 = 0
\end{equation}
  We will also need that
\begin{equation}
\label{E:D4}
\int Q^2 = 8\pi \,, \qquad \int Q^3 = 24\pi \,, \qquad \int x^2Q^2 = 8 \pi
\end{equation}
Recall that
$$E_0(Q) = -\frac12 \int Q \, HQ_x \, dx - \frac16 \int Q^3$$
Substituting \eqref{E:D3}, 
$$E_0(Q) = -\frac12 \int Q^2 + \frac1{12} \int Q^3$$
Plugging in \eqref{E:D4}, we obtain
\begin{equation}
\label{E:D5}
E_0(Q) = -2\pi
\end{equation}
Taking 
$$Q_{a, c}(x) = cQ(c(x-a))$$
we obtain
\begin{equation}
\label{E:D6}
\int Q_{a, c}^2 = c\int Q^2 = 8\pi c \,, \qquad \text{and} \qquad E_0(Q_{a, c}) = c^2 E_0(Q) = -2\pi c^2
\end{equation}
Let 
$$\mathcal{L} = I - H\partial_x - Q$$
Key properties of $\mathcal{L}$ follow by differentiating \eqref{E:scaled-sol} with respect to $a$ and with respect to $c$.  Specifically, we have
\begin{equation}
\label{E:D15}
\mathcal{L}Q_x = 0 \, \qquad \text{and} \quad \mathcal{L}(Q+xQ') = -Q
\end{equation}
Key spectral properties of $\mathcal{L}$ follow from the two identities
\begin{equation}
\label{E:LQLQ2}
\mathcal{L} Q = - \frac12 Q^2 \,, \qquad \mathcal{L} Q^2 = -2Q - Q^2
\end{equation}
Summary of derivation:
\begin{itemize}
\item direct computation gives $xQ' = \frac12 Q^2 - 2Q$
\item substitute into $\mathcal{L}(Q+xQ') = -Q$ to obtain $\mathcal{L}(-Q+\frac12Q^2) = -Q$
\item soliton equation gives $\mathcal{L}Q = (Q-H\partial_xQ -\frac12Q^2) - \frac12Q^2 = -\frac12Q^2$
\item add to get $\mathcal{L}(\frac12Q^2) = -Q - \frac12Q^2$
\end{itemize}
By \eqref{E:LQLQ2}
$$\mathcal{L}(2\alpha Q + \beta Q^2) = -2\beta Q - (\alpha+\beta)Q^2$$
By suitably selecting $\alpha$ and $\beta$, we can achieve the two eigenfunctions corresponding to $\lambda_+$ and $\lambda_-$ in the following proposition.
To achieve eigenfunctions, we need $\frac{\beta}{\alpha} = \frac{\alpha+\beta}{\beta}$, which yields a quadratic equation with solutions $\frac{\beta}{\alpha} = \frac12 \pm \frac12 \sqrt 5$.  If we take $\alpha =1$ and $\beta=\frac12+\frac12\sqrt{5}$ and 
$$e_- = 2Q + (\tfrac12+\tfrac12\sqrt{5})Q^2 \,, \qquad \lambda_- = - (\tfrac12+\tfrac12\sqrt{5})$$
then $\mathcal{L}e_- = \lambda_- e_-$.

\begin{proposition}[from Appendix of \cite{BBSSB}]
\label{P:Lspectral}
The operator $\mathcal{L}$ has exactly four eigenvalues
$$
\lambda_1= 1, \quad \lambda_0 =0, \quad \lambda_+ = \frac{-1+\sqrt{5}}{2}, \quad \lambda_- =\frac{-1-\sqrt{5}}{2}
$$
and a continuous spectrum $[1, +\infty)$.
Moreover, the corresponding eigenspaces are one-dimensional, the eigenfunction for $\lambda_0 =0$ is $ Q'$, and the (non normalized) eigenfunction for $\lambda_-$ is 
$$e_- = 2Q + \tfrac12(1+\sqrt{5}) Q^2$$
\end{proposition}

The above properties of $\mathcal{L}$ convert, via scaling, to corresponding properties of $\mathcal{L}_{a, c}$, where
$$\mathcal{L}_{a,c} = cI - H\partial_x - Q_{a,c}$$
For example, \eqref{E:D15} implies 
$$\mathcal{L}_{a,c}( \partial_a Q_{a,c}) =0 \,, \qquad \mathcal{L}_{a,c} (\partial_c Q_{a,c}) = -Q_{a,c}$$
In the case $a=0$, we simply denote the operator $\mathcal{L}_c = \mathcal{L}_{0,c}$.

For any nonlinear functional $\mathcal{A}: L^2(\mathbb{R})\to \mathbb{R}$ (perhaps densely defined), the Frechet derivative $\mathcal{A}'(u)\in L^2$ can be identified with a function, and the second derivative $\mathcal{A}''(u)$ can be regarded as an operator $L^2\to L^2$, by the identifications
$$\forall \; v\in L^2 \,, \quad \mathcal{A}'(u)(v) = \la \mathcal{A}'(u),v\ra$$
$$\forall \; v_1,v_2 \in L^2 \,, \quad \mathcal{A}''(u)(v_1,v_2) = \la \mathcal{A}''(u)v_1,v_2\ra$$
For our analysis, there are a few important functionals.  The free Hamiltonian is
$$E_0(u) = \frac12 \|D^{1/2} u \|_{L^2}^2 - \frac16 \int u^3$$
For the first and second derivatives, we have
$$E_0'(u) = Du - \frac12 u^2$$
$$E_0''(u) = D - u$$
We also consider the perturbed Hamiltonian 
$$E(u) = E_0(u) + \frac12 \int Vu^2$$
For the first and second derivatives, we have
$$E'(u) = E_0'(u) + Vu$$
$$E''(u) = E_0''(u) + V$$
The mass functional is
$$M(u) = \frac12 \|u\|_{L^2}^2$$
For the first and second derivatives
$$M'(u) = u$$
$$M''(u) = I$$
Consider the combined functional
$$Z_c(u) = cM(u) + E_0(u)$$
For the first and second derivatives (holding $c$ fixed), we have
$$Z_c'(u) = cu - H\partial_x u - \frac12 u^2$$
$$Z_c''(u) = c - H\partial_x  - u$$
Using the equation for $Q_{a,c}$, we obtain
$$Z_c'(Q_{a,c}) = 0$$
We note that 
$$\mathcal{L}_{a,c} = Z_c''(Q_{a,c})$$
With $\eta \defeq u - Q_{a,c}$, let
\begin{equation}
\label{E:Lyapunov}
\begin{aligned}
L_c(u) &= Z_c(u) -Z_c(Q_{a,c})- \la Z_c'(Q_{a,c}), \eta\ra \\
&= Z_c(u) - Z_c(Q_{a,c}) \\
&= \frac12 \la \mathcal{L}_{a,c} \eta, \eta \ra - \frac16 \int \eta^3
\end{aligned}
\end{equation}
i.e. $L_c(u)$ is the quadratic and higher order part of $Z_c(u)$ around the reference function $Q_{a,c}$.   This will be used as the Lyapunov functional in \S\ref{S:remainder}.

\section{Symplectic restriction heuristics}
\label{S:heuristics}

The phase space is $N=L^2(\mathbb{R}\to \mathbb{R})$ (real-valued $L^2$ functions on $\mathbb{R}$), with the inner product
$$
\la v_1,v_2 \ra = \int v_1 v_2 \, dx
$$
The energy is
\begin{equation}
\label{E:D7}
E(u) =  \frac12 \| D^{1/2}u\|_{L^2}^2 - \frac16 \int u^3 + \frac12 \int Vu^2
\end{equation}
which is a densely defined nonlinear map $N\to \mathbb{R}$.
With respect to the inner product $\la \cdot, \cdot, \ra$, $E'(u)$ is identified with an element of $N$ and $E''(u)$ is identified with a map $N\to N$, which are given explicitly by
$$E'(u) = Du - \frac12 u^2 + Vu = -Hu_x - \frac12 u^2 + Vu$$
$$E''(u) = -H\partial_x - u + V$$
Introduce the operator $J:N\to N$ given by $J = \partial_x$, which is skew with respect to $\la \cdot, \cdot \ra$.  This gives a symplectic form on $N$
$$\omega(v_1,v_2) = \la v_1, J^{-1}v_2 \ra$$
which is only densely defined.  The corresponding Hamiltonian flow is
$$\partial_t u = JE'(u)$$
which is precisely (pBO).  

Let us consider the \emph{soliton manifold} $M$ given by
$$M = \{ \; Q_{a,c} \; \}_{a\in \mathbb{R}, c>0}$$
where
$$Q_{a,c}(x) = cQ(c(x-a))$$
From \eqref{E:D3}, we obtain that $Q_{a,c}$ solves the equation
\begin{equation}
\label{E:D8}
-HQ_{a,c}' + cQ_{a,c} - \frac12 Q_{a,c}^2=0
\end{equation}
It is simpler to phrase some calculations using the group action
$$g_{a,c}f(x) = cf(c(x-a))$$
so that, in particular, $Q_{a,c} = g_{a,c}Q$.
To compute the restricted symplectic form, we need
$$\partial_a Q_{a,c} = - \partial_x Q_{a,c} = -c g_{a,c}\partial_x Q$$
$$\partial_c Q_{a,c} = c^{-1} g_{a,c} \partial_x (xQ) = c^{-2} \partial_x g_{a,c}(xQ)$$
It is also convenient to use the following (which follows by the change of variable $y=c(x-a)$)
$$\la g_{a,c}v_1, g_{a,c} v_2 \ra = c \la v_1,v_2\ra $$
We compute
\begin{align*}
\omega( \partial_a Q_{a,c}, \partial_c Q_{a,c}) &= \la \partial_a Q_{a,c}, J^{-1} \partial_cQ_{a,c}\ra = -c^{-1} \la g_{a,c} Q', g_{a,c} (xQ) \ra \\
&= -\la Q',xQ\ra = \frac12 \|Q\|_{L^2}^2 = 4\pi
\end{align*}
Thus, the restricted symplectic form is
$$i^*(\omega) = 4\pi da \wedge dc$$
where $i:M \to N$ is the inclusion.  This form is non degenerate, so $M$ is a \emph{symplectic submanifold}.  Let us now restrict the Hamiltonian $E$ to $M$.  
$$E(Q_{a,c}) = \frac12 \| D^{1/2} Q_{a,c}\|^2 - \frac16 \| Q_{a,c} \|_{L^3}^3+ \frac12 \int V Q_{a,c}^2 $$
To compute the first term, we substitute the soliton equation
$$\frac12 \| D^{1/2} Q_{a,c}\|^2 = -\frac12 \la HQ_{a,c}',Q_{a,c} \ra = -\frac12 c \| Q_{a,c}\|_{L^2}^2  + \frac14 \|Q_{a,c}\|_{L^3}^3$$
Hence
$$E(Q_{a,c}) =  -\frac12 c \| Q_{a,c}\|_{L^2}^2 + \frac1{12} \| Q_{a,c} \|_{L^3}^3 + \frac12 \int VQ_{a,c}^2$$
$$E(Q_{a,c}) = c^2( -\frac12  \| Q\|_{L^2}^2 + \frac1{12} \| Q \|_{L^3}^3) + \frac12 \int VQ_{a,c}^2$$
Substituting the values of the integrals \eqref{E:D4},
\begin{equation}
\label{E:D10}
E(Q_{a,c}) = -2\pi c^2 + \frac12 \int VQ_{a,c}^2
\end{equation}
Now we view the two-dimensional symplectic manifold $M$, with sympletic form $4\pi da\wedge dc$, and the Hamiltonian $E(Q_{a,c})$ (which is just a function $M\to \mathbb{R}$), as a two-dimensional Hamiltonian system.  The corresponding equations of motion are
\begin{equation}
\label{E:D9}
\left\{
\begin{aligned}
&\dot c = \frac{1}{4\pi} \partial_a E(Q_{a,c})\\
&\dot a = -\frac{1}{4\pi} \partial_c E(Q_{a,c})
\end{aligned}
\right.
\end{equation}
For this, we will want an asymptotic (as $h\to 0$) computation of $\frac12 \int V Q_{a,c}^2$.  Changing variable to $y=c(x-a)$, we obtain
$$\frac12 \int W(hx) c^2 Q(c(x-a))^2 \, dx = \frac12 c \int W(ha + hc^{-1}y) Q(y)^2 \, dy$$
$$=\frac12 c \int (W(ha) + \frac12 h^2c^{-2}y^2W''(ha) + O(h^3)y^3) Q(y)^2 \, dy$$
where the odd terms in the Taylor expansion are dropped since they have zero integral
(although actually the term $\int y^3Q(y)^2$ is not absolutely convergent, but since $\supp W$ is compact, the integral can be localized to $|x|\lesssim h^{-1}$, leaving only a $\log h^{-1}$ loss).
$$ = \frac12 cW(ha) \int Q^2 + \frac14 h^2c^{-1}W''(ha) \int y^2Q(y)^2 \, dy + O(h^3)$$
Substituting the values of the integrals \eqref{E:D4},
$$\frac12 \int W(hx) c^2 Q(c(x-a))^2 \, dx  = 4\pi cW(ha) + 2\pi c^{-1} W''(ha) h^2 + O(h^3)$$
Plugging into \eqref{E:D10}, we obtain
$$E(Q_{a,c}) = -2\pi c^2 + 4\pi cW(ha) + 2\pi c^{-1} W''(ha) h^2 + O(h^3)$$
Plugging in to \eqref{E:D9}, we get
$$
\left\{
\begin{aligned}
&\dot c = \frac{1}{4\pi} \partial_a E(Q_{a,c}) = chW'(ha) + \frac12 c^{-1}h^3W'''(ha) +O(h^4)\\
&\dot a = -\frac{1}{4\pi} \partial_c E(Q_{a,c}) = c- W(ha) + \frac12 c^{-2} h^2 W''(ha) + O(h^3)
\end{aligned}
\right.
$$
These indeed match, to leading order, the ODEs \eqref{E:ref-traj} describing the parameter dynamics in Theorem \ref{T:main}.

\section{Estimates for the Hilbert transform}
\label{S:estimates}

Below we provide two estimates for quadratic forms involving the Hilbert transform, Lemma \ref{L:co1}, \ref{L:co2}, that will later be needed in the proof of Lemma \ref{L:energycontrol}.

\begin{lemma}
\label{L:co1}
For $\chi \in C_c^\infty(\mathbb{R})$ and $0< h \leq 1$, we have
\begin{equation}
\label{E:co18}
\left| \int_y  \chi(hy)  \cdot w \cdot H \partial_y w \, dy \right| \lesssim \| w \|_{H_y^{1/2}}^2
\end{equation}
where the implicit constant depends on $\chi$ but is uniform in $h$.
\end{lemma}
\begin{proof}
By \cite[Theorem A.8]{KPV}, the fractional Leibniz rule states:  for $0<\alpha<1$, $0 \leq \alpha_1,\alpha_2 \leq \alpha$ with $\alpha = \alpha_1+\alpha_2$, $1<p,p_1,p_2<\infty$ with $\frac{1}{p} = \frac{1}{p_1}+\frac{1}{p_2}$, there holds
\begin{equation}
\label{E:FL}
\| D_x^\alpha(fg) - f D_x^\alpha g - g D_x^\alpha f \|_{L_x^p} \lesssim \|D_x^{\alpha_1} f\|_{L_x^{p_1}} \|D_x^{\alpha_2} g \|_{L_x^{p_2}}
\end{equation}
To prove \eqref{E:co18}, we write
$$
\int_y  \chi(hy)  \cdot w \cdot H \partial_y w \, dy = -\int_y  D_y^{1/2}[\chi(hy)  \cdot w] \cdot D_y^{1/2} w \, dy
$$
so by Cauchy-Schwarz,
\begin{equation}
\label{E:FL3}
\left| \int_y  \chi(hy)  \cdot w \cdot H \partial_y w \, dy \right| \leq \| D_y^{1/2}[\chi(hy)  \cdot w]\|_{L_y^2} \|D_y^{1/2} w\|_{L_y^2}
\end{equation}
We apply \eqref{E:FL} with $\alpha=\frac12$, $\alpha_1=\alpha_2 = \frac14$, and $p=2$, $p_1=p_2=4$, to obtain
$$
\| D_y^{1/2}[\chi(hy)  \cdot w]  - D_y^{1/2}[\chi(hy)] \cdot w - \chi(hy) \cdot D_y^{1/2} w \|_{L_y^2} \lesssim \| D_y^{1/4}[\chi(hy)] \|_{L_y^4} \| D_y^{1/4} w \|_{L_y^4}
$$
To the right side, we apply the Sobolev embedding $\|f\|_{L^4} \lesssim \| D^{1/4} f\|_{L^2}$, to obtain
\begin{equation}
\label{E:FL2}
\| D_y^{1/2}[\chi(hy)  \cdot w]\|_{L_y^2} \lesssim  
\begin{aligned}[t]
&\| D_y^{1/2}[\chi(hy)]\|_{L_y^\infty} \|w\|_{L_y^2} + \| \chi(hy) \|_{L_y^\infty} \| D_y^{1/2} w\|_{L_y^2} \\
&+ \| D_y^{1/2}[\chi(hy)] \|_{L_y^2} \| D_y^{1/2} w \|_{L_y^2}
\end{aligned}
\end{equation}
By the interpolation estimates, 
$$\| D_y^{1/2}[\chi(hy)]\|_{L_y^\infty} \lesssim \|\chi(hy) \|_{L_y^2}^{1/2} \|[\chi(hy)]'' \|_{L_y^2}^{1/2} \lesssim h^{1/2} \|\chi\|_{L^2}^{1/2} \|\chi''\|_{L^2}^{1/2} $$
$$\|D_y^{1/2}[\chi(hy)]\|_{L_y^2} \lesssim \| \chi(hy) \|_{L^2}^{1/2} \| [\chi(hy)]'\|_{L_y^2}^{1/2} \lesssim h^{1/2} \|\chi\|_{L^2}^{1/2} \|\chi'\|_{L^2}^{1/2}$$
inserted into \eqref{E:FL2} we obtain
$$
\| D_y^{1/2}[\chi(hy)  \cdot w]\|_{L_y^2} \lesssim \|w\|_{H_y^{1/2}}$$
where now the implicit constant depends on $\chi$, but is independent of $h$ for $0<h\leq 1$.  Plugging into \eqref{E:FL3}, we obtain the desired estimate.
\end{proof}

\begin{lemma}
\label{L:co2}
For $\chi \in C_c^\infty(\mathbb{R})$ and $0<h \leq 1$, we have
\begin{equation}
\label{E:co19}
\left| \int_y  \chi(hy)  \cdot  H \partial_y w \cdot  \partial_y w \, dy \right| \lesssim h^2 \|w\|_{L_y^2}^2 
\end{equation}
where the implicit constant depends on $\chi$ but is uniform in $h$.
\end{lemma}
\begin{proof}
By scaling, it suffices to take $h=1$.
We follow the beginning of the proof of Lemma 3 on p. 916 of Kenig \& Martel \cite{KM}, where it is observed that
\begin{equation}
\label{E:co22}
\int_y  \chi(hy)  \cdot  H \partial_y w \cdot  \partial_y w \, dy  = \iint_{y,y'} K(y,y') w(y) w(y') \, dy \, dy'
\end{equation}
where 
\begin{align}
\notag K(y,y') &=  - \frac{\partial^2}{\partial y \partial y'} \left( \frac{ \chi(y) - \chi(y')}{y-y'} \right) \\
\label{E:co23} & = \frac{ 2(\chi(y)-\chi(y')) - (\chi'(y)+\chi'(y'))(y-y')}{(y-y')^3}
\end{align}
Using Taylor formulas to third order for $\chi(y)$ and $\chi(y')$, we obtain that for each $y$, $y'$, there exists $y_1$, $y_2$ between $y$ and $y'$ such that
$$K(y,y') = \frac12 \frac{ \chi''(y')-\chi''(y)}{y-y'} + \frac16 (\chi'''(y_1)+\chi'''(y_2))$$
By the mean-value theorem, there exists $y_3$ between $y$ and $y'$ such that
$$K(y,y') =  \frac16 \chi'''(y_1) + \frac16\chi'''(y_2)+ \frac12 \chi'''(y_3) $$
and hence 
\begin{equation}
\label{E:co24}
|K(y,y')| \lesssim 1
\end{equation} 
for all $y,y'\in \mathbb{R}$.  

By decomposition into $|y-y'|\leq 1$ (in which case we appeal to \eqref{E:co24}) and $|y-y'|\geq 1$ (in which case we appeal to \eqref{E:co23}), we obtain
$$\| K\|_{L_y^\infty L_{y'}^1} \lesssim 1 \,, \qquad  \|K \|_{L_{y'}^\infty L_y^1} \lesssim 1$$
We complete the proof by estimating the right side of \eqref{E:co22} using the Schur test to obtain the bound
$$\| K\|_{L_y^\infty L_{y'}^1}^{1/2} \|K \|_{L_{y'}^\infty L_y^1}^{1/2} \|w\|_{L^2}^2$$

\end{proof}

\section{Energy estimate}
\label{S:energybd}

Below in Lemma \ref{L:en}, we state and prove the spectral lower bound on $\mathcal{L}$, following ideas of Weinstein \cite{Wei} and Fr\"ohlich et. al. \cite{FGJS}, and using the explicit spectral resolution of $\mathcal{L}$ provided in the Appendix of Bennett et. al \cite{BBSSB} quoted above as Prop. \ref{P:Lspectral}.  Lemma \ref{L:en} will be needed in the proof of Lemma \ref{L:energycontrol}.

\begin{lemma}[energy bound]
\label{L:en}
There exists $\kappa_1>1$ such that the following holds.  If, for some  $\frac12 \leq c \leq 2$, $w$ satisfies $\la w(y), cQ(cy) \ra =0$ and $\la w(y), cyQ(cy)\ra=0$, then 
$$\kappa_1^{-1} \|w \|_{H^{1/2}}^2 \leq \la \mathcal{L}_c w, w \ra  \leq \kappa_1 \|w\|_{H^{1/2}}^2$$
\end{lemma}

\begin{proof}
We treat the case $c=1$.  By elliptic regularity, it suffices to prove that
\begin{equation}
\label{E:lowerbd-1}
 \inf_{ \substack{ \la w, Q \ra=0 \\ \la w, xQ \ra=0 \\ \|w\|_{L^2}=1}} \la \mathcal{L}w, w \ra > 0
 \end{equation}
To prove \eqref{E:lowerbd-1}, it suffices to prove
\begin{equation}
\label{E:lowerbd-2}
\inf_{ \substack{ \la w, Q \ra=0 \\ \|w\|_{L^2}=1}} \la \mathcal{L}w, w \ra =0
\end{equation}
Indeed, suppose that \eqref{E:lowerbd-2} holds.  Then we know that the inf in \eqref{E:lowerbd-1} is $\geq 0$.  Thus, we can assume by contradiction that it is $=0$.  A minimizer $f_0$ satisfies the Euler-Lagrange equation
\begin{equation}
\label{E:lowerbd-3}
 \mathcal{L} f_0 = \mu f_0 + \nu Q + \omega xQ
 \end{equation}
Pairing \eqref{E:lowerbd-3} with $f_0$ yields $\mu=0$.  Pairing \eqref{E:lowerbd-3} with $Q'$ yields $\omega =0$.  Then $\mathcal{L}f_0 = \nu Q$.  But $Q$ is orthogonal to the kernel of $\mathcal{L}$, and hence there is a unique solution given by $f_0 = \nu \mathcal{L}^{-1}Q$.  We know that $\mathcal{L}(Q+xQ') = -Q$, so $f_0 = -\nu (Q+xQ')$.  However, the property $\la f_0, Q\ra =0$ yields that $\nu=0$.  Thus, $f_0=0$, which contradicts that $\|f_0\|_{L^2}=1$.  This concludes the proof that \eqref{E:lowerbd-2} implies \eqref{E:lowerbd-1}.

It remains to prove \eqref{E:lowerbd-1}.  Let $\alpha$ be the stated infimum, so that we aim to show that $\alpha =0$.  Since $w=\frac{Q'}{\|Q'\|_{L^2}}$ satisfies $\la w, Q \ra =0$ and $\|w\|_{L^2}=1$, we know that $\alpha \leq 0$.  We assume by contradiction that $\alpha<0$.     A minimizer $w_0$ satisfies the Euler-Lagrange equation
\begin{equation}
\label{E:lowerbd-4}
\mathcal{L}w_0 = \mu w_0 + \nu Q
\end{equation}
Pairing \eqref{E:lowerbd-4} with $w_0$ yields $\alpha = \mu$, we can now rewrite \eqref{E:lowerbd-4} as
\begin{equation}
\label{E:lowerbd-5}
\mathcal{L}w_0 = \alpha w_0 + \nu Q
\end{equation}
We know that $\mathcal{L}$ has one negative eigenvalue $\lambda_- = -\frac12-\frac12\sqrt{5}$ with corresponding eigenfunction $e_- = 2Q + \frac12(1+\sqrt{5})Q^2$.  Since 
$$\lambda_- = \inf_{\|w\|_{L^2}=1} \la \mathcal{L}w, w\ra$$
it follows that $\alpha \geq \lambda_-$.   If $\alpha=\lambda_-$, then pairing \eqref{E:lowerbd-5} with $e_-$ yields $0= \nu \la Q,e_-\ra$.  From the formula for $e_-$, we conclude that $\nu=0$, and so $w_0 = \frac{e_-}{\|e_-\|_{L^2}}$.  This contradicts that $\la w_0,Q\ra =0$.    Therefore
$$\lambda_- < \alpha<0$$
and we know that $\mathcal{L}-\alpha$ is invertible.  Returning to \eqref{E:lowerbd-5}, we can write
\begin{equation}
\label{E:lowerbd-7}
w_0 = \nu (\mathcal{L}-\alpha)^{-1}Q
\end{equation}
Pairing with $Q$ we obtain
\begin{equation}
\label{E:lowerbd-6}
0 = \nu \la (\mathcal{L}-\alpha)^{-1}Q,Q\ra 
\end{equation}
Consider the function
$$q(\lambda) =  \la (\mathcal{L}-\alpha)^{-1}Q,Q\ra$$
Since $q'(\lambda) = \|(\mathcal{L}-\alpha)^{-1}Q\|^2 >0$, the function $q(\lambda)$ is increasing.  Moreover $q(0) = \la \mathcal{L}^{-1}Q,Q\ra = -\la Q+xQ', Q\ra = - \frac12 \|Q\|^2 <0$.  Consequently  $q(\alpha)<0$.   Returning to \eqref{E:lowerbd-6}, we conclude that $\nu=0$.  Then \eqref{E:lowerbd-7} becomes $w_0=0$, contradicting that $\|w_0\|_{L^2}=1$.  

\end{proof}

\section{Setup}
\label{S:setup}

Theorem \ref{T:main} will follow from the following Prop. \ref{P:main} below. 

\begin{proposition} 
\label{P:main}
There exists a constant $\kappa >1$, $\mu>0$, and $0<h_0\ll 1$ such that the following holds.  Let $0< h \leq h_0$ and suppose initially
$$\| u_0(x) - Q_{0,1}(x) \|_{H_x^{1/2}} \leq h^{3/2}$$
Let $u$ solve (pBO) with initial condition $u_0$.  Then there exist $C^1$ parameters $(a(t),c(t))$ such that
\begin{equation}
\label{E:control'}
 \| u(x,t) - Q_{a(t),c(t)}(x) \|_{H_x^{1/2}} \leq \kappa h^{3/2}e^{\mu ht}
\end{equation}
for $0\leq t \leq T_0=h^{-1}\min( \frac14\mu^{-1} \ln h^{-1}, S_0)$ and 
  \begin{equation}
\label{E:ODE12}
\begin{aligned}
&|\dot { c}  - chW'( ha) | \lesssim \kappa^2 h^3 e^{2\mu ht}\\
&|\dot { a} -  c+ W( ha) +\frac12 c^{-2} h^2 W''(ha)| \lesssim \kappa^2 h^3 e^{2\mu ht}
\end{aligned}
\end{equation}
\end{proposition}

\begin{proof}[Proof that Prop. \ref{P:main} implies Theorem \ref{T:main}]
In the conclusion of Prop. \ref{P:main}, we are given $(a(t),c(t))$, from which we can define $(A(s),C(s))$ by $a(t)=h^{-1}A(ht)$ and $c(t) = C(ht)$.  Then $(A(s),C(s))$ satisfy
\begin{equation}
\label{E:ODE11}
\begin{aligned}
&|\dot { C}  - CW'( A)| \leq \kappa^2 h^2 e^{2\mu S}\\
&|\dot { A} -  C+ W( A) +\frac12 C^{-2} h^2 W''(A)| \leq \kappa^2 h^3 e^{2\mu S}
\end{aligned}
\end{equation}
on $0\leq s \leq  \min(\frac14\mu^{-1} \ln h^{-1},S_0)$.  Recall that $(\bar A(s), \bar C(s))$ were defined to solve \eqref{E:ref-traj} with $(\bar A(0),\bar C(0))=(0,1)$.    Now apply the Lemma \ref{L:gronwall} on ODE perturbation to conclude that
$| A - \bar A| \lesssim h^2 e^{2\mu s}$ and $|C - \bar C| \lesssim h^2e^{2\mu s}$, and thus $|a-\bar a| \lesssim h e^{2\mu hs}$ and $|c-\bar c| \lesssim h^2 e^{2\mu hs}$.   Thus Theorem \ref{T:main} follows.
\end{proof}

In the rest of the paper, we will prove Prop. \ref{P:main}.
Define the remainder $\eta$ according to 
\begin{equation}
\label{E:decomp}
u = Q_{a,c} + \eta
\end{equation}
imposing orthogonality conditions
\begin{equation}
\label{E:orth}
\la \eta, Q_{a,c} \ra =0 \,, \quad \la \eta, (x-a) Q_{a,c} \ra = 0
\end{equation}
An implicit function theorem argument shows that there exists a unique choice of $(a,c)$ so that these orthogonality conditions hold.  This is the \emph{definition} of the parameters $(a(t),c(t))$ and of the remainder $\eta$.

Starting with $\partial_t u = JE'(u)$, we substitute \eqref{E:decomp} to obtain
$$\partial_t (Q_{a,c} + \eta) = J E'( Q_{a,c} + \eta)$$
Using expansions
\begin{itemize}
\item $\partial_t Q_{a,c} = \dot a \partial_a Q_{a,c} + \dot c \partial_c Q_{a,c}$
\item $E'(u) = -H\partial_x u - \frac12 u^2 + Vu$
\item $E''(u) = -H\partial_x  - u +V$
\end{itemize}
we obtain the \emph{equation for the remainder}
\begin{equation}
\label{E:eta}
\partial_t \eta = - \dot a \partial_a Q_{a,c} - \dot c \partial_c Q_{a,c} + J E'(Q_{a,c}) + JE''(Q_{a,c}) \eta  - \frac12 \partial_x (\eta^2)
\end{equation}
The soliton part on the right side is simplified as
\begin{align*}
J E'(Q_{a,c}) &= \partial_x (- H\partial_x Q_{a,c} - \frac12Q_{a,c}^2 + W(hx) Q_{a,c}) \\
&= \partial_x ( -c Q_{a,c} + W(hx) Q_{a,c})
\end{align*}

The proof of Prop. \ref{P:main} is completed by bootstrapping Lemmas \ref{L:ODEcontrol} and \ref{L:energycontrol} in the next two sections.

\section{ODE control assuming remainder control}
\label{S:ODEcontrol}

\begin{lemma}[ODE control]
\label{L:ODEcontrol}
For each $\kappa \geq 1$ and for each $\mu>0$, there exists $0<h_0\ll 1$ such that the following holds.  Let $0< h \leq h_0$ and $T$ be any time so that $0\leq T \leq T_0 \defeq \min(\frac14\mu^{-1} \log h^{-1}, S_0)$.  If $\eta$ solves \eqref{E:eta} and $(a,c)$ satisfy, for all $0 \leq t\leq T$, the orthogonality conditions \eqref{E:orth} and we further assume 
\begin{equation}
\label{E:BS1}  
\| u(x,t) - Q_{a(t),c(t)}(x) \|_{H_x^{1/2}} \leq \kappa h^{\frac32}e^{\mu ht}
\end{equation}
then the ODE estimates hold ($ht=s$)
\begin{equation}
\label{E:BS2}
\begin{aligned}
&|\dot { c}  - chW'(ha)| \lesssim \kappa^2 h^3 e^{2\mu ht}\\
&|\dot { a} -  c+ W( ha) +\frac12 c^{-2} h^2 W''(ha)| \lesssim \kappa^2 h^3 e^{2\mu ht}
\end{aligned}
\end{equation}
on $0\leq t \leq T$.
\end{lemma}
\begin{proof}
Taking $\partial_t$ of the orthogonality condition $\la \eta, J^{-1} \partial_a Q_{a,c} \ra = - \la \eta, Q_{a,c} \ra =0$, we obtain
$$0 = \la \partial_t \eta, J^{-1} \partial_a Q_{a, c}\ra - \la \eta, \partial_t Q_{a, c}\ra$$
Substituting in the $\eta$ equation \eqref{E:eta} and expanding $\partial_tQ = \dot a \partial_a Q + \dot c \partial_c Q$,
\begin{equation}
\begin{split}
0 
& =  - \dot{a} \la  \partial_a Q_{a, c}, J^{-1} \partial_a Q_{a, c}   \ra - \dot{c} \la  \partial_c Q_{a, c}, J^{-1} \partial_a Q_{a, c}   \ra + \la J E'(Q_{a, c}), J^{-1} \partial_a Q_{a, c}  \ra \\
& + \la J E''(Q_{a, c}) \eta, J^{-1} \partial_a Q_{a, c}   \ra - \frac{1}{2} \la \partial_x (\eta^2), J^{-1} \partial_a Q_{a, c} \ra - \la \eta, \dot{a} \partial_a Q_{a, c}  \ra - \la  \eta, \dot{c} \partial_c Q_{a, c}  \ra \\ 
& =  I + II + III + IV + V + VI + VII \ . \\
\end{split}
\end{equation}
Noticing the skew-adjontness of $J^{-1}$, we obtain $I =0$. We compute $II$ as follows:
\begin{equation}
II = - \dot{c} \la  \partial_c Q_{a, c}, J^{-1} \partial_a Q_{a, c}   \ra  =  \dot{c} \la  \partial_c Q_{a, c},   Q_{a, c}   \ra  = \frac{1}{2} \dot{c} \partial_c \int Q^2_{a, c} dx  = 4 \pi \dot{c} \ . \\
\end{equation}
For $III$, using $E' (Q_{a, c}) = - c Q_{a, c} + W (hx)Q_{a, c}  $ and $\int  (x-a) Q^2_{a, c} (x) dx =0 $, we have
$$
III =  \la J E'(Q_{a, c}), J^{-1} \partial_a Q_{a, c}  \ra =  - \la  E'(Q_{a, c}),  \partial_a Q_{a, c}  \ra  = -  \partial_a E(Q_{a, c})$$
Using that $E(Q_{a, c}) = E_0(Q_{a, c}) + \frac12 \int W(hx) Q_{a, c}(x)^2 \, dx$, and (by \eqref{E:D6}) $E(Q_{a, c}) = -2\pi c^3$, 
$$III = -\frac 12 \partial_a \int W(hx) Q_{a, c}(x)^2 \, dx$$
Distributing $\partial_a$ onto $Q_{a, c}(x)^2$, converting $\partial_a = -\partial_x$, then integrating by parts\footnote{We note that by using the method employed to treat Term III' below, we can obtain an expression with accuracy $O(h^4)$}
\begin{align*}
III &= -\frac12 h\int W'(hx) Q_{a, c}(x)^2 \, dx\\
& = - \frac{1}{2} h \int  (W'(ha) + h W'' (ha) (x-a) + O(h^2)(x-a)^2 ) Q_{a, c}^2 (x) dx  \\
& = - \frac{1}{2} h  W'(ha) \int Q_{a, c}^2 (x) dx + O (h^3)  \\
& = - 4 \pi ch W'(ha)+ O (h^3) \ . 
\end{align*}
where we used \eqref{E:D6} and that $Q_{a, c}(x)$ is even around $x=a$.  For $IV$, we compute
\begin{equation}
\label{E:four} 
\begin{aligned}
IV &= \la J E''(Q_{a, c}) \eta, J^{-1} \partial_a Q_{a, c}   \ra   = - \la E''(Q_{a, c}) \eta,  \partial_a Q_{a, c}   \ra   \\
&= - \la  \eta,   E''(Q_{a, c}) \partial_a Q_{a, c}   \ra  = - \la  \eta,   \partial_a ( E'(Q_{a, c}) )  \ra   
\end{aligned}
\end{equation}
Using $E' (Q_{a, c}) = - c Q_{a, c} + W (hx) Q_{a, c}$, we have by Taylor expansion
\begin{align*}
\partial_a E'(Q_{a, c}) &= - c \partial_a Q_{a, c}(x) + W(hx) \partial_a Q_{a, c}(x) \\
&= (-c+W(ha)) \partial_a Q_{a, c}(x) + hW'(ha) (x-a) \partial_aQ_{a, c}(x) \\
& \qquad + h^2 \nu_{h,a}(x) \,(x-a)^2 \partial_a Q_{a, c}(x) \\
&= (-c+W(ha)) \partial_a Q_{a, c}(x) - hW'(ha) \partial_x [(x-a) Q_{a, c}(x)] \\
& \qquad + hW'(ha) Q_{a, c}(x) + h^2 \nu_{h,a}(x) \,(x-a)^2 \partial_a Q_{a, c}(x)
\end{align*}
where
$$\nu_{h,a}(x) = \int_0^1 W''(ha + h(x-a)\sigma)(1-\sigma) \, d\sigma $$
which satisfies $\|\nu_{h,a} \|_{L_x^\infty} \lesssim 1$ uniformly in $h$ and $a$.   Plugging into \eqref{E:four},
\begin{align*}
IV &= (c-W(ha)) \la \eta, \partial_a Q_{a, c} \ra + hW'(ha) \la \eta, \partial_x[(x-a)Q_{a, c}] \ra - hW'(ha) \la \eta, Q_{a, c} \ra \\
& \qquad - h^2 \la \eta, \nu_{h,a} (x-a)^2\partial_aQ_{a, c} \ra
\end{align*}
Using that $\partial_x[(x-a)Q_{a, c}]= c\partial_c Q_{a, c}(x)$ and $\la \eta, Q_{a, c} \ra =0$, we obtain
\begin{align*}
IV &= (c-W(ha)) \la \eta, \partial_a Q_{a, c} \ra + c hW'(ha) \la \eta, \partial_c Q_{a, c}\ra - h^2 \la \eta, \nu_{h,a} (x-a)^2\partial_aQ_{a, c} \ra
\end{align*}
Using $\|\eta\|_{L_x^2} \leq \kappa h^{3/2} e^{\mu ht}$ (which is \eqref{E:BS1}), $\|\nu_{h,a} \|_{L_x^\infty} \lesssim 1$, $\| (x-a)^2 \partial_aQ_{a, c} \|_{L_x^2} \lesssim 1$, we obtain
\begin{align*}
IV &= (c-W(ha)) \la \eta, \partial_a Q_{a, c} \ra + c hW'(ha) \la \eta, \partial_c Q_{a, c}\ra + O( \kappa h^{7/2} e^{\mu ht} )
\end{align*}
Bringing in VI and VII, we have
\begin{align*}
IV+VI+VII &= (-\dot a + c-W(ha)) \la \eta, \partial_a Q_{a, c} \ra + (-\dot c + c hW'(ha)) \la \eta, \partial_c Q_{a, c}\ra \\
&\qquad + O( \kappa h^{7/2} e^{\mu ht} )
\end{align*}
For V, using that $J^{-1}\partial_aQ_{a, c} = -Q_{a, c}$, $\|Q_{a, c}\|_{L_x^\infty} \leq 4c \leq 8$, and Cauchy-Schwarz,
$$|V| = | \la \eta \eta_x, Q_{a, c} \ra| \leq 8 \| \eta\|_{L_x^2} \|\eta_x \|_{L_x^2}$$
By \eqref{E:BS1},
$$|V| \leq 8 \kappa^2 h^3 e^{2\mu h t} $$

Putting the estimates of all the terms together, we have
\begin{equation}  
\label{E:ODE13}
\begin{aligned}
0 &=  4 \pi ( \dot{c} - ch W' (ha) ) + (-\dot{a} +c - W(ha) ) \la \eta, \partial_a Q_{a, c} \ra \\
& \qquad + (- \dot{c} + ch W' (ha) ) \la \eta, \partial_c Q_{a, c} \ra + O(\kappa^2 e^{2\mu ht} h^3)
\end{aligned}
\end{equation}

On the other hand, we take $\partial_t$ of the orthogonality condition $\la \eta, J^{-1} \partial_c Q_{a,c} \ra  =0$ to obtain
$$0 = \la \partial_t \eta, J^{-1}\partial_c Q_{a, c}\ra + \la \eta, \partial_t [J^{-1}\partial_cQ_{a, c}] \ra $$
Recalling that $\partial_x[ (x-a)Q_{a, c}(x)] = c\partial_c Q_{a, c}(x)$,
\begin{align*}
0 &= \la \partial_t \eta, J^{-1}\partial_c Q_{a, c}\ra + \la \eta, \partial_t [ c^{-1}(x-a)Q_{a, c}] \ra \\
&= \la \partial_t \eta, J^{-1}\partial_c Q_{a, c}\ra -c^{-2}\dot c \la \eta, (x-a)Q_{a, c} \ra  - c^{-1} \dot a \la \eta, Q_{a, c}\ra \\
& \qquad + c^{-1} \dot a \la \eta, (x-a)\partial_a Q_{a, c} \ra + c^{-1} \dot c \la \eta, (x-a)\partial_c Q_{a, c} \ra \\
&= \la \partial_t \eta, J^{-1}\partial_c Q_{a, c}\ra + c^{-1} \dot a \la \eta, (x-a)\partial_a Q_{a, c} \ra + c^{-1} \dot c \la \eta, (x-a)\partial_c Q_{a, c} \ra
\end{align*}
where we have used the orthogonality conditions.  Substituting in the $\eta$ equation \eqref{E:eta}
\begin{equation}
\begin{split}
0 & =  - \dot{a} \la  \partial_a Q_{a, c}, J^{-1} \partial_c Q_{a, c}   \ra - \dot{c} \la  \partial_c Q_{a, c}, J^{-1} \partial_c Q_{a, c}   \ra + \la J E'(Q_{a, c}), J^{-1} \partial_c Q_{a, c}  \ra \\
& \qquad + \la J E''(Q_{a, c}) \eta, J^{-1} \partial_c Q_{a, c}   \ra - \frac{1}{2} \la \partial_x (\eta^2), J^{-1} \partial_c Q_{a, c} \ra + c^{-1} \dot a \la \eta, (x-a)\partial_a Q_{a, c} \ra \\
& \qquad + c^{-1} \dot c \la \eta, (x-a)\partial_c Q_{a, c} \ra \\ 
& =  I' + II' + III' + IV' + V' + VI' +VII' \ . \\
\end{split} 
\end{equation}

Again, from the skew-adjontness of $J^{-1}$, we obtain $II' =0$. For $I'$ we compute:
\begin{equation}
\begin{split}
I' &= - \dot{a} \la  \partial_a Q_{a, c}, J^{-1} \partial_c Q_{a, c}   \ra  =  \dot{a} \la J^{-1} \partial_a Q_{a, c},   \partial_c Q_{a, c}   \ra  \\
&  = - \dot{a} \la   Q_{a, c},    \partial_c Q_{a, c}   \ra = - \frac{1}{2} \dot{a} \partial_c \int Q^2_{a, c} dx  = - 4 \pi \dot{a} \ . \\
\end{split}
\end{equation}
where we have used \eqref{E:D6}.  For $III'$, we observe
$$III' = \la JE'(Q_{a, c}), J^{-1} \partial_c Q_{a, c} \ra = - \la E'(Q_{a, c}), \partial_c Q_{a, c} \ra = - \partial_c E(Q_{a, c}) $$
By \eqref{E:D6}, $E(Q_{a, c}) = -2\pi c^2 + \frac12 \int W(hx) Q_{a, c}(x)^2 \, dx$, and so
$$III' = 4\pi c - \frac12 \int W(hx) \partial_c ( Q_{a, c}(x)^2 ) \, dx$$
Note that
$$\partial_c ( Q_{a, c}^2 ) = 2 Q_{a, c} \, \partial_c Q_{a, c} = 2Q_{a, c} \,  \partial_x[ c^{-1}(x-a) Q_{a, c} ] = 2c^{-1} Q_{a, c}^2 + c^{-1} (x-a) \partial_x [Q_{a, c}^2]$$
Substituting and integrating by parts,
\begin{equation}
\label{E:odestuff1}
\begin{aligned}
III' &= 4\pi c - c^{-1} \int W(hx) Q_{a, c}(x)^2 \, dx + \frac12 c^{-1}\int  \partial_x[  W(hx) (x-a) ] Q_{a, c}(x)^2 \, dx \\
&= 4 \pi c + c^{-1} \int [-\frac12 W(hx) + \frac12 hW'(hx)(x-a)] Q_{a, c}(x)^2 \, dx
\end{aligned}
\end{equation}
Split the integral into $|x-a|< h^{-1}$ and $|x-a|> h^{-1}$.  By the decay of $Q_{a, c}(x)$, the second region contributes only $O(h^3)$.  For the inner region, we use the Taylor expansions with integral remainder
\begin{align*}
W(hx) &= W(ha) + hW'(ha)(x-a)+ \frac12 h^2 W''(ha)(x-a)^2 \\
&\qquad + \frac16 h^3W'''(ha)(x-a)^3 + \frac1{24} h^4 (x-a)^4 \nu_{h,a}^0(x)
\end{align*}
where 
$$\nu_{h,a}^0(x) = 4\int_0^1 W''''(ha+h(x-a)\sigma) (1-\sigma)^3 \, d\sigma$$
and
\begin{align*}
W'(hx) &= W'(ha) + hW''(ha)(x-a)+ \frac12 h^2 W'''(ha)(x-a)^2  + \frac16 h^3 (x-a)^3\nu_{h,a}^1(x)
\end{align*}
where 
$$\nu_{h,a}^1(x) = 3\int_0^1 W'''(ha+h(x-a)\sigma) (1-\sigma)^2 \, d\sigma$$
Combining,
\begin{equation}
\label{E:odestuff2}
\begin{aligned}
\indentalign -\frac12 W(hx) + \frac12 hW'(hx)(x-a) \\
&= -\frac12 W(ha)   + \frac14 h^2 W''(ha) (x-a)^2 \\
&\qquad + \frac16 h^3 W'''(ha)(x-a)^3 + h^4 (x-a)^4 \nu_{h,a}(x)
\end{aligned}
\end{equation}
where
$$\nu_{h,a}(x) = -\frac{1}{48} \nu_{h,a}^0(x) + \frac{1}{12} \nu_{h,a}^1(x)$$
which is uniformly bounded independently of $h$ and $a$.  Plugging \eqref{E:odestuff2} into \eqref{E:odestuff1} truncated to $|x-a|<h^{-1}$, we obtain
\begin{align*}
III' &= 4\pi c - \frac12 c^{-1}W(ha) \int_{|x-a|<h^{-1}} Q_{a, c}(x)^2 \, dx \\ 
& \qquad + \frac12 c^{-1} h W'(ha) \int_{|x-a|<h^{-1}} (x-a) Q_{a, c}(x)^2 \, dx \\
& \qquad +\frac14 c^{-1} h^2W''(ha) \int_{|x-a|<h^{-1}} (x-a)^2 Q_{a, c}(x)^2 \, dx \\
& \qquad + \frac{1}{6} c^{-1} h^3 W'''(ha) \int_{|x-a|<h^{-1}} (x-a)^3 Q_{a, c}(x)^2 \, dx \\
& \qquad + h^4 c^{-1} \int_{|x-a|<h^{-1}} (x-a)^4 \nu_{h, a}(x) Q_{a, c}(x)^2 \, dx + O(h^3)
\end{align*}
Using that $Q_{a, c}(x)^2$ is even around $x=a$, the third and fifth terms drop out.  Also $(x-a)^4Q_{a, c}(x)^2$ is uniformly bounded, and $\nu_{h, a}(x)$ is uniformly bounded, so the integral in the sixth term yields a factor $h^{-1}$ (the length of the domain of integration), and this term becomes another contribution to the $O(h^3)$ error.  Thus we are reduced to 
\begin{align*}
III' &= 4\pi c - \frac12 c^{-1}W(ha) \int_{|x-a|<h^{-1}} Q_{a, c}(x)^2 \, dx \\
& \qquad +\frac14 c^{-1} h^2W''(ha) \int_{|x-a|<h^{-1}} (x-a)^2 Q_{a, c}(x)^2 \, dx + O(h^3)
\end{align*}
Owing to the decay of $Q_{a, c}(x)^2$, the integrals above can be replaced with integrals over all of $\mathbb{R}$ at the expense of $O(h^3)$ error
\begin{align*}
III' &= 4\pi c - \frac12 c^{-1}W(ha) \int Q_{a, c}(x)^2 \, dx \\
& \qquad +\frac34 c^{-1} h^2W''(ha) \int (x-a)^2 Q_{a, c}(x)^2 \, dx + O(h^3)
\end{align*}
Also by \eqref{E:D4}, 
$$\int Q_{a, c}(x)^2 \, dx = c \int Q(x)^2 \,dx = 8\pi c$$ 
and 
$$\int (x-a)^2 Q_{a, c}(x)^2 \, dx = c^{-1}\int x^2 Q(x)^2 \, dx = 8\pi c^{-1}$$   
Substituting,
$$III' = 4\pi c - 4\pi W(ha) + 2\pi h^2 W''(ha) c^{-2} + O(h^3)$$

For $IV'$,
\begin{equation}
\begin{split}
IV'
& = \la J E''(Q_{a, c}) \eta , J^{-1} \partial_c Q_{a, c} \ra   = - \la E''(Q_{a, c}) \eta ,  \partial_c Q_{a, c} \ra \\
& = - \la  \eta ,  E''(Q_{a, c}) ( \partial_c Q_{a, c} ) \ra = - \la \eta , \partial_c (E' (Q_{a, c})) \ra 
\end{split}
\end{equation}
Using that $E'(Q_{a, c}) = -c Q_{a, c} + W(hx) Q_{a, c}$, we obtain
$$IV' = \la \eta, Q_{a, c} \ra + c\la \eta, \partial_c Q_{a, c}\ra  - \la \eta, W(hx) \partial_c Q_{a, c} \ra$$
By the orthogonality	conditions, the first term drops away.  The third term is divided into an inner region $|x-a|<h^{-1}$ and an outer region $|x-a|> h^{-1}$,
$$IV' =  c\la \eta, \partial_c Q_{a, c}\ra  - \la \eta, W(hx) \partial_c Q_{a, c} \ra_{|x-a|<h^{-1}} - \la \eta, W(hx) \partial_c Q_{a, c} \ra_{|x-a|>h^{-1}}$$
For the outer region, we have
$$| \la \eta, W(hx) \partial_c Q_{a, c} \ra_{|x-a|>h^{-1}} | \lesssim \| \eta\|_{L_x^2} \|W\|_{L_x^\infty} \| \partial_c Q_{a, c} \|_{L_{|x-a|>h^{-1}}^2} \lesssim h^{3/2} \|\eta\|_{L_x^2}$$
By \eqref{E:BS1},
$$| \la \eta, W(hx) \partial_c Q_{a, c} \ra_{|x-a|>h^{-1}} |  \lesssim \kappa h^3 e^{\mu ht }$$
and thus we are reduced to
$$IV' =  c\la \eta, \partial_c Q_{a, c}\ra  - \la \eta, W(hx) \partial_c Q_{a, c} \ra_{|x-a|<h^{-1}} +O( \kappa h^3 e^{\mu ht })$$
For the inner region, we use Taylor expansion with integral remainder to second order
$$W(hx) = W(ha) + hW'(ha)(x-a) + \frac12 h^2 \nu_{h, a}(x) (x-a)^2$$
where
$$\nu_{h, a}(x) = 2 \int_0^1 W''(ha + h(x-a)\sigma) (1-\sigma) \, d\sigma$$
Substituting,
\begin{align*}
IV' &= c  \la \eta, \partial_c Q_{a, c}\ra - W(ha) \la \eta, \partial_c Q_{a, c}\ra_{|x-a|<h^{-1}} \\
& \qquad  - hW'(ha) \la \eta, (x-a) \partial_c Q_{a, c} \ra_{|x-a|<h^{-1}}  \\
& \qquad - \frac12 h^2 W''(ha)\la \eta, (x-a)^2 \partial_c Q_{a, c} \ra_{|x-a|<h^{-1}}\\
& \qquad +O( \kappa h^3 e^{\mu ht })
\end{align*}
In the second and third term, the integrals can be extended to all of $\mathbb{R}$ at the expense of additional contribution to the $O( \kappa h^3 e^{\mu ht })$ error.  In the fourth term, we use that $(x-a)^2 \partial_c Q_{a, c} = O(1)$ and apply Cauchy-Schwarz (using that the length of the interval of integration is $h^{-1}$) to obtain
$$|\frac12 h^2 W''(ha)\la \eta, (x-a)^2 \partial_c Q_{a, c} \ra_{|x-a|<h^{-1}}|  \lesssim h^{3/2} \|W''\|_{L_x^\infty} \|\eta\|_{L_x^2} \lesssim \kappa h^3 e^{\mu ht}$$
This yields
$$
IV' = (c -W(ha)) \la \eta, \partial_c Q_{a, c}\ra  - hW'(ha) \la \eta, (x-a) \partial_c Q_{a, c} \ra +O( \kappa h^3 e^{\mu ht })
$$
For VI', note that
$$(x-a)\partial_a Q_{a, c} = -(x-a)\partial_xQ_{a, c} = -\partial_x[ (x-a)Q_{a, c} ] + Q_{a, c} = -c\partial_c Q_{a, c} + Q_{a, c}$$
By the orthogonality conditions,
$$VI' = c^{-1}\dot a \la \eta, (x-a)\partial_a Q_{a, c}\ra = - \dot  a \la \eta, \partial_c Q_{a, c} \ra$$
Thus
\begin{align*}
IV'+VI'+VII' &= (-\dot a + c -W(ha)) \la \eta, \partial_c Q_{a, c}\ra  \\
& \qquad +(c^{-1}\dot c - hW'(ha)) \la \eta, (x-a) \partial_c Q_{a, c} \ra +O( \kappa h^3 e^{\mu ht })
\end{align*}
Finally, we have
$$V' = - \frac{1}{2} \la \partial_x (\eta^2), J^{-1} \partial_c Q_{a, c} \ra = \frac12 \la \eta^2 , \partial_c Q_{a, c} \ra$$
Since  $\partial_c Q_{a, c} $ is bounded,   \eqref{E:BS1} implies
$$|V'| \lesssim \|\eta \|_{L_x^2}^2 \lesssim \kappa^2 h^3 e^{2\mu ht}$$
Substituting,
\begin{equation}
\label{E:ODE14}
\begin{aligned}
0 &= 4\pi (-\dot a + c - W(ha) +\frac12 h^2 W''(ha) c^{-2}) + (-\dot a + c -W(ha)) \la \eta, \partial_c Q_{a, c}\ra  \\
& \qquad +(c^{-1}\dot c - hW'(ha)) \la \eta, (x-a) \partial_c Q_{a, c} \ra +O( \kappa^2 h^3 e^{2\mu ht })
\end{aligned}
\end{equation}
Combining \eqref{E:ODE13} and \eqref{E:ODE14}, and using $\|\eta\|_{L_x^2} \lesssim \kappa h^{3/2} e^{\mu ht}$, we obtain a matrix equation
$$(4 \pi I - A(\eta))
\begin{bmatrix}
\dot c - ch W'(ha) \\ \dot a -c +W(ha) - \frac12 h^2 W''(ha) c^{-2} 
\end{bmatrix} = O(\kappa^2 h^3 e^{2\mu ht})
$$
where
$$A(\eta)=\begin{bmatrix} \la \eta, \partial_c Q_{a, c}\ra & \la \eta, \partial_a Q_{a, c} \ra \\ c^{-1} \la \eta, (x-a) \partial_c Q_{a, c}\ra & - \la \eta, \partial_c Q_{a, c} \ra \end{bmatrix}$$
By \eqref{E:BS1}, $|A(\eta)| \ll 1$, and thus standard inversion completes the proof.
\end{proof}

\section{Remainder control assuming ODE control}
\label{S:remainder}

Taylor expand $W(hx)$ around $x=a$ to obtain
$$W(hx) = W(ha) + hW'(ha)(x-a) + e_2$$
and use
$$\partial_aQ_{a,c} = - \partial_x Q_{a,c} \,, \qquad \partial_c Q_{a,c} = c^{-1} \partial_x[ (x-a) Q_{a,c} ]$$
Substituting into \eqref{E:eta},
$$\partial_t \eta = 
\begin{aligned}[t]
&(\dot a - c + W(ha))\partial_x Q_{a,c} + (-\dot c c^{-1} + hW'(ha))\partial_x[(x-a)Q_{a,c}]  \\
&+ \partial_x(e_2Q_{a,c}) + JE''(Q_{a,c}) \eta  - \frac12 \partial_x (\eta^2)
\end{aligned}
$$
where
$$e_2(x) = W(hx)-W(ha) - hW'(ha)(x-a)$$
We need to recenter the equation.  For this, define $w(y,t)$ by
$$\eta(x,t) = w(y,t) \,, \qquad y = x-a$$
$$\mathcal{L}_c = c - H\partial_y - cQ(cy)$$
Then
$$\partial_t \eta =  - \dot a  \partial_y w  + \partial_t w \,, \qquad E''(Q_{a,c}) \eta = (\mathcal{L}_c - c + V) w$$
Substituting, the equation for $\eta$ converts to the following equation for $w$
$$\partial_t w = 
\begin{aligned}[t]
&(\dot a - c + W(ha)) cQ'(cy) + (-c^{-1}\dot c + hW'(ha)) (cyQ(cy))' \\
&+ \partial_y(e_2 cQ(cy)) + \partial_y\mathcal{L}_c w + (\dot a - c + W(ha)) \partial_y w \\
&+ \partial_y [(W(ha+ hy)-W(ha)) w]  - \frac12 \partial_y (w^2)
\end{aligned}
$$
and now
$$e_2(y) = W(ha+hy)-W(ha)-hW'(ha)y$$

\begin{lemma}[energy remainder control]
\label{L:energycontrol}
For each $\kappa \geq 1$ and each $\mu>0$, there exists $0<h_0\ll 1$ such that the following holds.
Suppose $0 < h \leq h_0$ and $0\leq T \leq T_0 \defeq h^{-1}\min( \frac14\mu^{-1}\log h^{-1}, S_0)$. 
If $\eta$ solves \eqref{E:eta} and $(a,c)$ satisfy, for all $0 \leq t\leq T$, the orthogonality conditions \eqref{E:orth} and we further assume 
\begin{equation}
\label{E:BS3init}
\| u(x,0) - Q_{a(0),c(0)}(x) \|_{H_x^{1/2}} \leq h^{3/2}
\end{equation}
and for all $0 \leq t\leq T$
\begin{equation}
\label{E:BS3}
\| u(x,t) - Q_{a(t),c(t)}(x) \|_{H_x^{1/2}} \leq \kappa h^{3/2}e^{\mu ht}
\end{equation}
then for all $0 \leq t\leq T$
\begin{equation}
\label{E:BS5}
 \| u(x,t) - Q_{a(t),c(t)}(x) \|_{H_x^{1/2}} \leq \alpha h^{3/2}(1+ \kappa \mu^{-1/2} e^{\mu ht})\\
\end{equation}
where the constant $\alpha\geq 1$ in \eqref{E:BS5} is independent of $\kappa$ and $\mu$.  
\end{lemma}
Before we prove this, we explain how to select $\kappa\geq 1$ and $\mu>0$ in terms of the given $\alpha>0$ in order to obtain a useful implication, and complete the proof of Proposition \ref{P:main}.  

\begin{proof}[Proof of Prop. \ref{P:main}]
Let  $\kappa = 4 \alpha$ and $\mu = \kappa^2$.  Then the right side of \eqref{E:BS5} becomes
$$\alpha h^{3/2}(1+ \kappa \mu^{-1/2} e^{\mu ht}) = \tfrac14 \kappa h^{3/2} (1 + e^{\mu ht}) \leq \tfrac12 \kappa h^{3/2} e^{\mu ht}$$
With this choice of $\kappa$ and $\mu$, we now let $T$ be the \emph{maximal} time so that \eqref{E:BS3} holds.  Since $\kappa>1$, \eqref{E:BS3init} holds, and $\|\eta (t)\|_{H^{1/2}}$ is continuous in time, we know that $T>0$.   Moreover, by the maximality and continuity of $\|\eta(t)\|_{H^{1/2}}$, if $T<T_0$ then equality holds $\|\eta(T)\|_{H^{1/2}} = \kappa h^{3/2}e^{\mu hT}$.  But Lemma \ref{L:energycontrol} applies for $0\leq t \leq T$, and thus in particular  \eqref{E:BS5} implies $\|\eta(T)\|_{H^{1/2}} \leq \frac12 \kappa h^{3/2}e^{\mu hT}$, a contradiction.   Hence in fact $T=T_0$, and we conclude that \eqref{E:BS3} holds for all $0\leq t \leq T_0$, concluding the proof of Proposition \ref{P:main}.
\end{proof}

\begin{proof}[Proof of Lemma \ref{L:energycontrol}]
From Lemma \ref{L:ODEcontrol}, we know that
\begin{equation} 
\label{E:BS6}
\begin{aligned}
&| - c^{-1} \dot { c}  - h W'( ha) | \lesssim  \kappa^2 h^3 e^{2\mu ht}\\
&|\dot { a} -  c + W( ha)  | \lesssim \kappa^2 h^2 e^{2\mu ht}
\end{aligned}
\end{equation}

Firstly we compute an upper bound on $\la \mathcal{L}_c w,w \ra  $ as follows.
Noting that $\mathcal{L}_c$ is self-adjoint, we have
$$\partial_t \la w , \mathcal{L}_c w \ra  = 2 \la w_t, \mathcal{L}_c w \ra \ .$$
Plugging in the equation for $w$, we have 
\begin{equation}
\label{E:BS31}
\begin{split}   
 \partial_t \la \mathcal{L}_c w , w  \ra 
&  =     2  \int_y \big[  (\dot{a} - c + W(ha)) cQ'(cy) + (-c^{-1} \dot{c} + h W'(ha) ) (cyQ(cy))' \\
&\qquad\qquad\qquad + \partial_y (e_2 cQ(cy))  + \partial_y \mathcal{L}_c w +  (\dot{a} - c + W(ha)) \partial_y w \\
&\qquad\qquad \qquad + \partial_y  [ ( W(ha+hy) - W(ha)) w  ] - \frac{1}{2} \partial_y (w^2)   \big]  \mathcal{L}_c w \, dy  \\
& : =   I + II + III + IV + V + VI + VII \ . \\
\end{split}
\end{equation}
We compute $I$ and $II$, using the self-adjointness of $\mathcal{L}_c$,
\begin{equation} 
\begin{split}
I   &=(\dot{a} - c + W(ha))  \la cQ'(cy),  \mathcal{L}_c w     \ra   \\
&=  (\dot{a} - c + W(ha))  \la  \mathcal{L}_c ( cQ'(cy)),   w    \ra   \\
&= 0 \ ,  
\end{split}
\end{equation}
\begin{equation} 
\begin{split}
II  
&  =   (-c^{-1} \dot{c} + h W'(ha) )    \la (cyQ(cy))' ,  \mathcal{L}_c w     \ra \\
&  =  (-c^{-1} \dot{c} + h W'(ha) )  \la  \mathcal{L}_c (cyQ(cy))',   w    \ra  \\
& =   (-c^{-1} \dot{c} + h W'(ha) )  \la  - c^2 Q(cy),   w   \ra  \\
&= 0 
\end{split}
\end{equation}
by the orthogonality conditions.

For $\text{III}$, note that
\begin{equation}
\label{E:termIII-1}
e_2(y) = e_3(y)+ \frac12 h^2 W''(ha) y^2
\end{equation}
where
$$e_3(y) = W(h(y+a)) - W(ha) - hW'(ha) y - \frac12 h^2W''(ha) y^2$$
which, by the integral form of the remainder in Taylor's expansion, has the form
$$e_3(y) =  h^3 y^3 q(hy)$$
where
$$q(z) = \int_0^1 W'''(zs+ha)) (1-s)^2 \, ds$$
From this expression, we see that $q(z)$ is smooth and by integration by parts, $|q^{(k)}(z)| \lesssim_k \la z \ra^{-k-1}$ for all $k\geq 0$ and hence 
\begin{equation}
\label{E:q-bound}
\| \la z \ra^k q^{(k)}(z) \|_{L^2_z} \lesssim_k  1 \,, \qquad \text{for all } k\geq 0 \,.
\end{equation}

Upon substituting \eqref{E:termIII-1}, we obtain
$$\text{III} = h^2  W''(ha) \la  \mathcal{L}_c\partial_y (y^2Q_c) , w \ra  + 2 \la  \mathcal{L}_c \partial_y(e_3 Q_c) , w \ra $$
Since $y^2Q_c(y) = 4c^{-1} - c^{-2} Q_c(y)$, we have $\mathcal{L}_c \partial_y (y^2Q_c) = 0$, and thus the first term drops out leaving
$$\text{III} =  2 \la  \mathcal{L}_c \partial_y(e_3 Q_c) , w \ra $$
Using that $ e_3 (y) = h^3 y^3 q(hy)$, we obtain
$$\partial_y (e_3Q_c) = 3 h^3 q(hy) \, y^2 Q_c(y) + h^3  \, (hy) q'(hy)  \, y^2 Q_c(y) + h^3 q(hy) \, y^3 (Q_c)'(y)$$
Thus
\begin{align*}
\| \partial_y (e_3Q_c) \|_{L^2_y} 
&\lesssim h^3 \| q(hy) \|_{L^2} \| y^2 Q_c(y) \|_{L^\infty} + h^3 \| (hy) q'(hy) \|_{L^2} \| y^2 Q_c(y) \|_{L^\infty} \\
& \qquad \qquad + h^3 \| q(hy) \|_{L^2} \| y^3 (Q_c)'(y)\|_{L^\infty}
\end{align*}
By rescaling in the $L^2$ norm terms and using \eqref{E:q-bound}
$$\| \partial_y (e_3Q_c) \|_{L^2_y} \lesssim h^{5/2}$$
Moreover, it can be checked that
$$\| H\partial_y^2 (e_3Q_c) \|_{L^2_y} \lesssim \| \partial_y^2 (e_3Q_c) \|_{L^2_y} \lesssim h^{7/2}$$
and consequently
$$\| \mathcal{L}_c \partial_y (e_3Q_c) \|_{L_y^2} \lesssim h^{5/2}$$
Thus, by \eqref{E:BS3},
$$| III | \lesssim \| \mathcal{L}_c \partial_y (e_3Q_c) \|_{L_y^2} \|w\|_{L_y^2}\lesssim \kappa h^4 e^{\mu h t}$$

For $IV$, we notice that $ \la \partial_y  \mathcal{L}_c w, \mathcal{L}_c w \ra =0$.

For $V$, we estimate,
\begin{equation*}
\begin{split}
V   &=    (\dot{a} - c + W(ha))  \la \partial_y w ,  \mathcal{L}_c w  \ra  \\
& =   (\dot{a} - c + W(ha))  ( \la \partial_y w , (- H \partial_y +c  )  w  \ra   +  \la  \partial_y w , - Q_c (y) w \ra  )  \\
& =   (\dot{a} - c + W(ha))    \la  \partial_y w , - Q_c (y) w \ra  
\end{split}
\end{equation*}
where we have used that that the first term contains skew-adjoint operators.  By integration by parts,
$$V  =  \frac12  (\dot{a} - c + W(ha))  \int_y Q'_c (y)  w^2 dy$$
and hence, by \eqref{E:BS3}, \eqref{E:BS6},
$$ | V | \leq   \frac12  |\dot{a} - c + W(ha)|  \|w\|^2_{L^2} \lesssim \kappa^2 h^5 e^{4\mu ht}$$
Using that $\kappa^2 h e^{\mu h t} \leq 1$, this becomes
$$ | V | \leq   \frac12  |\dot{a} - c + W(ha)|  \|w\|^2_{L^2} \lesssim  h^4 e^{2\mu ht}$$
For $VI$, 
\begin{equation}
\begin{split}
VI  
& =    \la  \mathcal{L}_c w , \partial_y  [ ( W(ha+hy) - W(ha)) w  ]   \ra  \\
& =    \la  ( - H \partial_y + c - Q_c ) w , \partial_y  [ ( W(ha+hy) - W(ha)) w  ]   \ra   \ . \\
\end{split}
\end{equation}
We have
\begin{equation}
\begin{split}
&    \la  c  w , \partial_y  [ ( W(ha+hy) - W(ha)) w  ]   \ra   \\
& =  \la  c  w , w \partial_y   ( W(ha+hy) - W(ha))   \ra  +   \la  c  w , \partial_y w     ( W(ha+hy) - W(ha))   \ra  \\
& =  \la  c  w , w \partial_y   ( W(ha+hy) - W(ha))   \ra  + \frac{1}{2}  \la  c  ( W(ha+hy) - W(ha)), \partial_y (w^2)   \ra   \\
& =   \la  c  w , w \partial_y   ( W(ha+hy) - W(ha))   \ra   - \frac{1}{2}  \la  c \partial_y ( W(ha+hy) - W(ha)),  w^2   \ra \\
& \lesssim    h \cdot \|w\|^2_{L^2}   \\
& \lesssim  \kappa^2 h^4 e^{2\mu ht}
\end{split}
\end{equation}
where we used \eqref{E:BS3}.  Also,
\begin{equation}
\begin{split}
&    \la  Q_c  w , \partial_y  [ ( W(ha+hy) - W(ha)) w  ]   \ra   \\
& =   \la  Q_c  w , w \partial_y   ( W(ha+hy) - W(ha))   \ra   +    \la  Q_c  w , \partial_y w     ( W(ha+hy) - W(ha))   \ra   \\
& =    \la  Q_c  w , w \partial_y   ( W(ha+hy) - W(ha))   \ra   + \frac{1}{2}   \la  Q_c  ( W(ha+hy) - W(ha)), \partial_y (w^2)   \ra   \\
& =    \la  Q_c  w , w \partial_y   ( W(ha+hy) - W(ha))   \ra   - \frac{1}{2}   \int_y (W(ha +hy) - W(ha)) Q'_c w^2   \\
& \lesssim     h  \|w\|^2_{L^2}   \\
& \lesssim \kappa^2  h^4 e^{2\mu ht} \\
\end{split}
\end{equation}
In the computation above we made use of the localization effect of $Q_c$. Next,
\begin{equation}
\begin{split}
&  \la  - H \partial_y  w , \partial_y  [ ( W(ha+hy) - W(ha)) w  ]   \ra    \\
& =   \la  - H \partial_y  w , w \partial_y   ( W(ha+hy) - W(ha))   \ra   +    \la  - H \partial_y  w , \partial_y w     ( W(ha+hy) - W(ha))   \ra   \\
& \leq  h  | \la   H \partial_y  w , W'(ha) w  \ra |   +   | \la   H \partial_y  w , \partial_y w     ( W(ha+hy) - W(ha))   \ra  |  \ . \\
\end{split}
\end{equation}
For the first term, we use Lemma \ref{L:co1} with $\chi (hy) = W'(ha) $ to obtain
$$ h  | \la   H \partial_y  w , W'(ha) w  \ra |  \lesssim h \|w\|^2_{H^{1/2}}  \lesssim \kappa^2 h^4e^{2ht} \ . $$
For the second term, we use Lemma \ref{L:co2} with $\chi (hy) = W(ha+hy) - W(ha) $ to obtain
$$  | \la   H \partial_y  w , \partial_y w     ( W(ha+hy) - W(ha))   \ra  |  \lesssim h^2  \|w\|^2_{H^{1/2}}   \lesssim \kappa^2 h^4 e^{2\mu h t}  \ . $$
Hence 
$$ VI \lesssim \kappa^2 h^4 e^{2\mu h t}    \ .  $$
Notice that in \eqref{E:BS31}, we have estimated all terms except $VII$.

Next, we compute $\frac{1}{3}\partial_t \int_y w^3 dy $. We have
\begin{equation}
\begin{split}   \label{E:BS32}
\frac{1}{3}\partial_t \int_y w^3 dy
& =    \int_y  w^2 w_t   dy   \\
&  =     \int_y \big[  (\dot{a} - c + W(ha)) cQ'(cy) + (-c^{-1} \dot{c} + h W'(ha) ) (cyQ(cy))' \\
& + \partial_y (e_2 cQ(cy))  + \partial_y \mathcal{L}_c w +  (\dot{a} - c + W(ha)) \partial_y w \\
&  + \partial_y  [ ( W(ha+hy) - W(ha)) w  ] - \frac{1}{2} \partial_y (w^2)   \big]  w^2 dy   \\
& : =   I' + II' + III' + IV' + V' + VI' + VII'  \ . \\
\end{split}
\end{equation} 
By \eqref{E:BS3} and \eqref{E:BS6}, we estimate
\begin{equation}
I' \lesssim \|w\|^2_{L^2} h^2  \lesssim \kappa^2 h^4 e^{2\mu ht} \ ,
\end{equation}
Similarly $II' \lesssim \kappa^2 h^4 e^{2\mu ht}$, and for $III'$, using that $\|\partial_y (e_2 Q_c) \|_{L_y^\infty} \lesssim h^2$, we obtain $III' \lesssim  \kappa^2 h^4 e^{2\mu ht}$.  Noticing that $\la \partial_yw, w^2 \ra =0$ and $\la \frac{1}{2} \partial_y (w^2), w^2 \ra = 0$, we obtain $V' = VII' =0$. 
For $VI'$, we compute
\begin{equation}
\begin{split}
VI' 
& =      \la  w \partial_y   ( W(ha+hy) - W(ha)),  w  \ra    +      \la   \partial_y w   ( W(ha+hy) - W(ha))  ,  w  \ra   \\
& \lesssim h     \int_y w^3 dy  +      \frac{1}{3} \la W(ha+hy)- W(ha), \partial_y (w^3) \ra  \\
& \lesssim h  \|w\|^3_{L^3}   +    \frac{1}{3} \la \partial_y [ W(ha+hy)- W(ha)] , w^3 \ra  \\
& \lesssim   \kappa^3 h^{11/2} e^{3\mu ht} 
\end{split}
\end{equation}
Using that $h^{1/2} e^{3\mu h t} \lesssim 1$, we conclude
$$VI' \lesssim h^5$$
Notice that in \eqref{E:BS32}, we have estimated all terms except $IV'$.  However, note that $VII$ from \eqref{E:BS31} and $IV'$ from \eqref{E:BS32} combined to give
\begin{equation}
\label{E:BS10}
VII-IV' =0
\end{equation}

Subtracting \eqref{E:BS31} by \eqref{E:BS32}, we appeal to the above estimates and the cancelation \eqref{E:BS10} to conclude that, under the assumptions of \eqref{E:BS3} and \eqref{E:BS6},
\begin{equation}
\label{E:BS11}
\left| \partial_t  \left( \la \mathcal{L}_c w, w\ra - \frac13 \int w^3 \, dy \right) \right| \lesssim \kappa^2 h^4 e^{2\mu ht}
\end{equation}
Integrating in time, we obtain
$$\la \mathcal{L}_c w(t), w(t)\ra \leq 
\begin{aligned}[t]
&\la \mathcal{L}_c w(0), w(0)\ra + \frac13 \int w(t)^3 \, dy - \frac13 \int w(0)^3 \, dy \\
&+ \beta \kappa^2 \mu^{-1}h^3 (e^{2\mu h t} -1)
\end{aligned}$$
for some constant $\beta>0$ independent of $\kappa$ and $\mu$.   We have
$$\| u\|_{L^3}^3 \lesssim \|w\|_{H^1}^3 \lesssim \kappa^3 e^{3\mu ht} h^{9/2} \leq \kappa^2 \mu^{-1} h^3 e^{2\mu h t}$$
since $t$ is bounded so that $\kappa h^{3/2} e^{\mu ht} \leq \mu^{-1}$.  This gives
$$\la \mathcal{L}_c w(t), w(t)\ra \leq \la \mathcal{L}_c w(0), w(0)\ra + (\beta+1) \kappa^2 \mu^{-1}h^3 e^{2\mu h t}
$$
By Lemma \ref{L:en}, there exists $\kappa_1\geq 1$ so that
$$\kappa_1^{-1}\| w(t)\|_{H^{1/2}}^2 \leq \kappa_1 \|w(0)\|_{H^{1/2}}^2 + (\beta+1) \kappa^2 \mu^{-1}h^3 e^{2\mu h t}
$$
By \eqref{E:BS3} evaluated at $t=0$, we obtain
$$\| w(t)\|_{H^{1/2}}^2 \leq \kappa_1^2 h^3 + \kappa_1(\beta+1) \kappa^2 \mu^{-1}h^3 e^{2\mu h t}
$$
from which \eqref{E:BS5} follows.
\end{proof}

\section{ODE perturbation theory}

\begin{lemma}[Gronwall]
\label{L:gronwall}
Suppose $X, \bar X :\mathbb{R} \to \mathbb{R}^d$ solve
\begin{align*}
&\dot X(s) = f(X(s)) + h^2 g(X,s) \\
&\dot{\bar X}(s) = f(\bar X(s))
\end{align*}
with the same initial condition $X(0)=\bar X(0)$,
where $f: \mathbb{R}^d \to \mathbb{R}^d$, $ g: \mathbb{R}^{d+1} \to \mathbb{R}^d$.  Suppose that the $d\times d$ matrix $f'(X)$ is uniformly bounded:  for all $X\in \mathbb{R}^d$, 
$$\| f'(X)\|_{\ell^2} \leq \kappa$$
where $\ell^2$ is the square sum norm on the $d^2$ entries of the matrix.  Then
$$|X(s)-\bar X(s)|^2 \leq  h^4 \int_0^s e^{-(2\kappa+1)(s-s')}|g(X(s'),s')|^2 \, ds'$$
\end{lemma}
\begin{proof}
Let $V(s) = X(s)-\bar X(s)$.  Then ($| \bullet |$ is the usual square sum norm on $\mathbb{R}^d$)
\begin{equation}
\label{E:ODE1}
\frac{d}{ds} |V|^2 = 2V\dot V = 2V\cdot (f(X)-f(\bar X)) + 2h^2 V\cdot g(X,s)
\end{equation}
We have
$$f(X) - f(\bar X) = \int_{\sigma=0}^1 \frac{d}{d\sigma}[ f(\bar X+ \sigma V)] \, d\sigma = \left( \int_{\sigma=0}^1 f'(\bar X+ \sigma V) \, d\sigma \right)V$$
Then by Cauchy-Schwarz,
$$| f(X) - f(\bar X)| \leq \kappa |V|$$
Substituting this into \eqref{E:ODE1}, and using that $2h^2 V \cdot g(X,s) \leq |V|^2 + h^4|g(X,s)|^2$, we obtain
$$\frac{d}{ds} |V|^2 \leq (2\kappa+1) |V|^2 + h^4|g(X,s)|^2$$
The standard integrating factor method completes the proof. 
\end{proof}
In our application,
$$X = \begin{bmatrix} A \\ C \end{bmatrix} \,, \qquad f(X) = \begin{bmatrix}
C-W(A) \\ CW'(A) \end{bmatrix}$$
Then
$$f'(X) = \begin{bmatrix} -W'(A)  & 1 \\ CW''(A) & W'(A) \end{bmatrix}$$
Since $\frac12\leq C \leq 2$, this is uniformly bounded.

\end{document}